\documentclass[11pt,reqno]{amsart}

\usepackage[utf8]{inputenc}

\usepackage{lineno}
\usepackage{amsaddr}
\usepackage{a4wide}
\usepackage{hyperref}

\usepackage{tikz}
\usetikzlibrary{arrows}

\usepackage{marginnote}

\usepackage{amsmath,amsthm,latexsym,xspace,amscd,amssymb,xy}               
\usepackage{color}
\usepackage{enumerate}

\newtheorem{theorem}{Theorem}[section]
\newtheorem{lemma}[theorem]{Lemma}
\newtheorem{corollary}[theorem]{Corollary}
\newtheorem{proposition}[theorem]{Proposition}

\theoremstyle{definition}
\newtheorem{definition}[theorem]{Definition}
\newtheorem{example}[theorem]{Example}
\newtheorem{remark}[theorem]{Remark}
\newtheorem{question}[theorem]{Question}

\mathchardef\mhyphen="2D

\author{Daniel Lännström}
\address{Department of Mathematics and Natural Sciences,
Blekinge Institute of Technology,
SE-37179 Karlskrona, Sweden}

\email{{\scriptsize daniel.lannstrom@bth.se}}

\date{\today}



\title[Induced quotient group gradings of epsilon-strongly graded rings]{Induced quotient group gradings of epsilon-strongly graded rings}

\begin{document}
\begin{abstract}
Let $G$ be a group and let $S=\bigoplus_{g \in G} S_g$ be a $G$-graded ring. Given a normal subgroup $N$ of $G$, there is a naturally induced $G/N$-grading of $S$. It is well-known that if $S$ is strongly $G$-graded, then the induced $G/N$-grading is strong for any $N$. The class of epsilon-strongly graded rings was recently introduced by Nystedt, Öinert and Pinedo as a generalization of unital strongly graded rings. We give an example of an epsilon-strongly graded partial skew group ring such that the induced quotient group grading is not epsilon-strong. Moreover, we give necessary and sufficient conditions for the induced $G/N$-grading of an epsilon-strongly $G$-graded ring to be epsilon-strong. Our method involves relating different types of rings equipped with local units ($s$-unital rings, rings with sets of local units, rings with enough idempotents) with generalized epsilon-strongly graded rings. 

\end{abstract}

\maketitle

\pagestyle{headings}

\section{Introduction}
Let $G$ be an arbitrary group with neutral element $e$. Our rings will be associative but not necessarily unital. Recall that a \emph{$G$-grading} of a ring $S$ is a family of additive subgroups $\{ S_g \}_{g \in G}$ of $S$ such that (i) $S=\bigoplus_{g \in G}S_g$ and (ii) $S_g S_h \subseteq S_{gh}$ for all $g,h \in G$. If the stronger condition $S_g S_h = S_{gh}$ holds for all $g, h \in G$, then the $G$-grading is called \emph{strong}. The $S_g$'s are called \emph{homogeneous components}. The \emph{principal component} $S_e$, i.e. the homogeneous component corresponding to the neutral element $e$, is a subring of $S$.  In general, there are many different $G$-gradings of a fixed ring $S$. However, as is common in the literature, we will write `$S$ is a $G$-graded ring' when we consider $S$ together with some implicit $G$-grading.   

We now recall the construction of the induced quotient grading. Let $S=\bigoplus_{g \in G} S_g$ be a $G$-graded ring and let $N$ be a normal subgroup of $G$. There is a natural way to define a new, induced $G/N$-grading of $S$. Let, 
\begin{equation*}
S_C :=\bigoplus_{g \in C} S_g, \qquad \forall \, C \in G/N.
\end{equation*}
It is straightforward to check that (i) $S=\bigoplus_{C \in G/N} S_C$ and (ii) $S_C S_{C'} \subseteq S_{C C'}$ for any $C, C' \in G/N$. The $G/N$-grading $\{ S_C \}_{C \in G/N}$ of $S$ is called the \emph{induced $G/N$-grading}. Note that the principal component of this $G/N$-grading is $S_{e N}=S_N$.  It is well-known that the `strongness' of the grading is preserved under this construction: If the original $G$-grading $\{ S_g \}_{g \in G}$ is strong, then the induced $G/N$-grading $\{ S_C \}_{C \in G/N}$ is strong (cf. Proposition \ref{prop:1}). This property is  of particular importance for the theory of strongly group graded rings initiated by Dade (see Section \ref{sec:prel}). 

The class of \emph{epsilon-strongly $G$-graded rings} was  introduced by Nystedt, Öinert and Pinedo \cite{nystedt2016epsilon} as a generalization of unital strongly $G$-graded rings.  A $G$-grading $\{S_g \}_{g \in G}$ of $S$ is \emph{epsilon-strong} if any only if, for every $g \in G$, there is an element $\epsilon_g \in S_g S_{g^{-1}}$ such that for all $s \in S_g$ the equations $\epsilon_g s = s = s \epsilon_{g^{-1}}$ hold (see \cite[Prop. 7]{nystedt2016epsilon}). In that case we say that $S$ is \emph{epsilon-strongly $G$-graded}.
It is natural to ask if `epsilon-strongness' is preserved by the induced quotient grading. 
This question was the starting point of this paper.
\begin{question}
If $\{ S_g \}_{g \in G}$ is epsilon-strong, is then $\{ S_C \}_{C \in G/N}$ epsilon-strong?
\label{question:intro}
\end{question}
We will give an example of an epsilon-strongly graded unital partial skew group ring such that the induced quotient group grading is not epsilon-strong (Example \ref{ex:1}). Our main result is a characterization of when the induced quotient group grading is epsilon-strong in terms of idempotents of the principal component (Theorem \ref{thm:main1}). 

To be able to formulate our main theorem, we first need some notation. For any ring $R$, let $E(R)$ denote the set of idempotents of $R$. There is a partial order on $E(R)$ defined by $a \leq b \iff a=ab =ba.$ Let $\vee$ and $\wedge$ denote the least upper bound and greatest lower bound with respect to this ordering. In the case where $a, b \in E(R)$ commute, it can be proven that $a \vee b = a + b-ab$ and $a \wedge b = ab$. For any set $F \subseteq E(R)$, we write $\bigvee F$ for the $\vee$-closure of $F$; i.e. $a, b \in \bigvee F \implies a \vee b \in \bigvee F$ provided that the upper bound $a \vee b$ exists. Recall (see \cite[Prop. 5]{nystedt2016epsilon}) that if $\{ S_g \}_{g \in G}$ is epsilon-strong, then the elements $\epsilon_g$ are central idempotents of the principal component $S_e$.
Our main result reads as follows:

\begin{theorem}
Let $S$ be an epsilon-strongly $G$-graded ring and let $N$ be a subgroup of $G$.
The induced $G/N$-grading $\{ S_C \}_{C \in G/N}$ is epsilon-strong if and only if, for every $C \in G/N$, there is some element $\chi_C \in E(S_N)$ such that $f \leq \chi_C$ for all $f \in \bigvee\{ \epsilon_g \mid g \in C  \}$.
\label{thm:main1}
\end{theorem}
%

Our method to prove Theorem \ref{thm:main1} involves generalizations of epsilon-strongly graded rings, which we call \emph{Nystedt-Öinert-Pinedo graded rings} (see Section \ref{sec:2}). The first generalization is the class of \emph{nearly epsilon-strongly $G$-graded rings} introduced by Nystedt and Öinert in \cite{nystedt2017epsilon} (see also \cite{nystedt2018epsilon}). Inspired by their definition, we will introduce two additional families of graded rings: \emph{essentially and virtually epsilon-strongly graded rings} (see Definition \ref{def:nystedt-gradings}). The proof of Theorem \ref{thm:main1} will amount to showing that any induced quotient group grading of an epsilon-strongly graded ring is essentially epsilon-strong (Theorem \ref{thm:main}). 


This paper will hopefully be the beginning of a larger effort to develop a theory of epsilon-strongly graded rings similar to the theory of strongly graded rings (cf. Section \ref{sec:prel}). Seemingly unrelated to the present investigation, the induced quotient group grading construction has been studied independently in \cite{johnson2012commutative} and \cite{sehgal2003graded}.

\smallskip

Below is an outline of the rest of this paper:

\smallskip

In Section \ref{sec:prel}, we give some additional background and a precise formulation of the problems considered in this paper.

In Section \ref{sec:2}, we define and prove some basic results about Nystedt-Öinert-Pinedo graded rings. In particular, we prove that only unital rings admit epsilon-strong gradings (Proposition \ref{prop:epsilon_unital}).

In Section \ref{sec:3}, we prove that Leavitt path algebras are virtually epsilon-strongly $G$-graded (Proposition \ref{prop:lpa_is_virtually}). This class will be an important source of examples. 



In Section \ref{sec:4}, we investigate the induced quotient group gradings of Nystedt-Öinert-Pinedo graded rings. We prove that the induced $G/N$-grading of a nearly epsilon-strongly $G$-graded ring is also nearly epsilon-strong (Proposition \ref{prop:nearly_induced}). Under certain assumptions, the same conclusion also holds for essentially epsilon-strongly graded rings (Corollary \ref{cor:essentially_induced}) and virtually epsilon-strongly graded rings (Proposition \ref{prop:virtually_induced}). As a special case, we get that the induced $G/N$-grading of a Leavitt path algebra is also virtually epsilon-strong (Corollary \ref{cor:lpa_induced}). Finally, we establish our main results: Theorem \ref{thm:main} and Theorem \ref{thm:main1}.

In Section \ref{sec:5}, we introduce a special type of epsilon-strong $G$-gradings called epsilon-finite gradings. This class has the property that for any normal subgroup $N$ of $G$, the induced $G/N$-grading is epsilon-finite (Proposition \ref{prop:epsilon_finite}). Moreover, one-sided noetherianity of the principal component $S_e$ is a sufficient condition for $S$ to be epsilon-strongly $G$-graded (Theorem \ref{thm:induced_epsilon}). 

In Section \ref{sec:6}, we give an example of an epsilon-strongly $\mathbb{Z}$-graded ring such that the induced $\mathbb{Z}/2\mathbb{Z}$-grading is not epsilon-strong (Example \ref{ex:1}).

In Section \ref{sec:7}, we consider induced quotient group gradings of epsilon-crossed products (see Section \ref{sub:epsilon_cross}). We give sufficient conditions for the induced $G/N$-grading of a unital partial skew group ring to give an epsilon-crossed product (Proposition \ref{prop:partial_crossed}).

\section{The induced quotient group grading functor}
\label{sec:prel}


Broadly speaking, we aim to investigate how the theory of strongly group graded rings
relates to epsilon-strongly graded ring. The systematic study of strongly group graded rings began with Dade's seminal paper \cite{dade1980group}. In that paper, he took a functorial approach and studied certain functors defined on the categories of strongly group graded rings and modules. Most notable is the celebrated Dade's Theorem, which asserts that a $G$-graded ring $S$ is strongly graded if and only if the category of graded modules over $S$ is equivalent to the category of modules over $S_e$. The following introductory example from \cite{dade1980group} shows the relation to classical Clifford theory. Let $G$ be an arbitrary group and recall that the complex group ring $\mathbb{C}[G] = \bigoplus_{g \in G} \mathbb{C} \delta_g$ is strongly $G$-graded. Furthermore, let $N$ be a normal subgroup of $G$. The induced quotient group grading $\mathbb{C}[G] = \bigoplus_{C \in G/N} S_C$ is strong.  Note that $S_N = \bigoplus_{g \in N} \mathbb{C} \delta_g = \mathbb{C}[N]$. By Dade's Theorem, there is an equivalence of categories between graded $\mathbb{C}[G]$-modules and $\mathbb{C}[N]$-modules. This example motivates us to take a functorial approach and to study the induced quotient group gradings of epsilon-strongly graded rings. 

In the remainder of this section, we will introduce further notation to precisely formulate the problems considered in this paper.


\subsection{Category of epsilon-strongly graded rings}

Let $G \mhyphen \mathrm{RING}$ denote the category of rings equipped with a $G$-grading. More precisely, the objects of $G \mhyphen \mathrm{RING}$ are pairs $(S, \{ S_g \}_{g \in G})$ where $S$ is a ring and $\{ S_g \}_{g \in G}$ is a $G$-grading of $S$. The morphisms are ring homomorphisms that respect the gradings. To make this more precise, let $(S, \{ S_g \}_{g \in G})$ and $(T, \{ T_g \}_{g \in G})$ be objects in $G \mhyphen \mathrm{RING}$. The ring homomorphism $\phi \colon S \to T$ is called \emph{$G$-graded} if $\phi(S_g) \subseteq T_g$ for each $g \in G$. The class $\text{hom}( (S, \{ S_g \}_{g \in G}), (T, \{ T_g \}_{g \in G}) )$ consists of the $G$-graded ring homomorphisms $S \to T$. By the definition of $G \mhyphen \mathrm{RING}$, it is straightforward to see that the category of strongly $G$-graded rings, which we denote by $G \mhyphen\mathrm{STRG}$, is a full subcategory of $G \mhyphen\mathrm{RING}$. We will later work with other subclasses of $G$-gradings, which similarly corresponds to full subcategories of $G \mhyphen \mathrm{RING}$. 

Next, we recall (see e.g. \cite[pg. 3]{nastasescu2004methods}) the definition of the \emph{induced quotient grading functor}. With the notation above, let $N$ be a normal subgroup of $G$ and let $\{ S_C \}_{C \in G/N}$, $\{ T_C \}_{C \in G/N}$ be the induced $G/N$-gradings of $S$ and $T$ respectively. If $\phi \colon S \to T$ is a $G$-graded homomorphism, then $\phi(S_C) \subseteq T_C$ for each $C \in G/N$. Hence, $\phi \colon S \to T$ is $G/N$-graded with respect to the induced $G/N$-gradings. This implies that the induced quotient group grading construction defines a functor,
\begin{align*}
U_{G/N} \colon G \mhyphen\mathrm{RING} & \to G/N \mhyphen \mathrm{RING}, \\
(S, \{ S_g \}_{g \in G}) & \mapsto (S, \{ S_C \}_{ C \in G/N }), \\
\text{hom}((S, \{ S_g \}_{g \in G}), (T, \{ T_g \}_{g \in G})) \ni \phi & \mapsto \phi \in \text{hom}((S, \{S_C \}_{C \in G/N}), (T, \{ T_C \}_{C \in G/N})).
\end{align*}


It is well-known that `strongness' is preserved by the induced quotient group grading: 

\begin{proposition}(cf. \cite[Prop. 1.2.2]{nastasescu2004methods}) Let $G$ be an arbitrary group and let $N$ be any normal subgroup of $G$. The functor $U_{G/N}$ restricts to the subcategory $G \mhyphen \mathrm{STRG}$. In other words, the functor,
\begin{equation*}
U_{G/N} \colon G \mhyphen \mathrm{STRG} \to G/N \mhyphen \mathrm{STRG},
\end{equation*} 
is well-defined.
\label{prop:1}
\end{proposition}

We denote the category of epsilon-strongly $G$-graded rings by $G \mhyphen \epsilon \mathrm{STRG}$. The objects of this category are epsilon-strongly $G$-graded rings and the morphisms are $G$-graded ring homomorphisms. The basic problem of this paper (cf. Question \ref{question:intro}) can be reformulated as:
\begin{question}
For which objects $(S, \{ S_g \}_{g \in G}) \in \text{ob}(G \mhyphen\epsilon \mathrm{STRG})$ do we have that, 
\begin{equation}
U_{G/N} ( (S, \{ S_g \}_{g \in G})) = (S, \{ S_C \}_{C \in G/N}) \in \text{ob}(G/N \mhyphen \epsilon \mathrm{STRG})?
\label{eq:11}
\end{equation}
In particular, is the restriction of $U_{G/N}$ to $G \mhyphen \epsilon \textrm{STRG}$ well-defined?
\label{question:2}
\end{question}
We will give an example of an epsilon-strongly $G$-graded ring such that (\ref{eq:11}) does not hold (Example \ref{ex:1}). In other words, the functor $U_{G/N}$ does not restrict to $G \mhyphen \epsilon \mathrm{STRG}$. Note that Theorem \ref{thm:main1} provides a complete answer to Question \ref{question:2}, i.e. a characterization of $(S, \{ S_g \}_{g \in G}) \in \text{ob}(G \mhyphen \epsilon \mathrm{STRG})$ such that (\ref{eq:11}) holds. 

\subsection{Epsilon-crossed products}
\label{sub:epsilon_cross}
Let $S$ be a strongly $G$-graded ring. Recall (see e.g. \cite[pg. 2]{nastasescu2004methods}) that $S$ is called an \emph{algebraic crossed product} if $S_g$ contains an invertible element for each $g \in G$. The class of epsilon-crossed products was introduced by Nystedt, Öinert and Pinedo in \cite{nystedt2016epsilon} as an `epsilon-analogue' of the classical algebraic crossed products. Let $G$ be an arbitrary group and let $S$ be an arbitrary epsilon-strongly $G$-graded ring. In other words, take an arbitrary object $(S, \{ S_g \}_{g \in G}) \in \text{ob}(G \mhyphen\epsilon \mathrm{STRG}).$ Recall (see \cite[Def. 30]{nystedt2016epsilon}) that an element $s \in S_g$ is called \emph{epsilon-invertible} if there exists some element $t \in S_{g^{-1}}$ such that $st=\epsilon_g$ and $ts=\epsilon_{g^{-1}}$. Furthermore, recall (see \cite[Def. 32]{nystedt2016epsilon}) that $(S, \{S_g \}_{g \in G})$ is called an \emph{epsilon-crossed product} if there is an epsilon-invertible element in $S_g$ for all $g \in G$. Let $G \mhyphen \epsilon \mathrm{CROSS}$ denote the category of epsilon-crossed products. The morphisms are $G$-graded ring homomorphisms. It is straightforward to show that $G \mhyphen \epsilon \mathrm{CROSS}$ is a full subcategory of $G \mhyphen \epsilon \mathrm{STRG}$.

For an algebraic crossed-product, the induced quotient group grading gives an algebraic crossed product (see e.g. \cite[Prop. 1.2.2]{nastasescu2004methods}). It is natural to ask when the induced quotient grading of an epsilon-crossed product gives an epsilon-crossed product. This is better formulated in terms of the functor $U_{G/N}$:
\begin{question}
For which objects $(S, \{ S_g \}_{g \in G}) \in \text{ob}(G \mhyphen \epsilon \mathrm{CROSS})$ do we have that, $$U_{G/N}((S, \{ S_g \}_{g \in G}) = (S, \{ S_C \}_{C \in G/N}) \in \text{ob}(G/N \mhyphen \epsilon \mathrm{CROSS})?$$ 
\label{question:1}
\end{question}
The author has not been able to answer Question \ref{question:1} in full generality. However, in Section 7, we will provide examples of epsilon-crossed products $(S, \{ S_g \}_{g \in G})  \in \text{ob}(G \mhyphen \epsilon \mathrm{CROSS})$ such that $(S, \{ S_C \}_{ C \in G/N}) \in \text{ob}(G/N \mhyphen \epsilon \mathrm{CROSS})$ and examples such that $(S, \{ S_C \}_{ C \in G/N}) \not \in \text{ob}(G/N \mhyphen\epsilon \mathrm{CROSS})$.

\section{Nystedt-Öinert-Pinedo graded rings}
\label{sec:2}


The purpose of this section is to introduce two new generalizations of epsilon-strongly graded rings. To this end, we recall several different ways in which a non-unital ring might have approximate multiplicative identity elements (also known as local units). We refer the reader to \cite{nystedt2018unital} for a detailed survey of these definitions. A ring $R$ has a \emph{set of local units} $E$ if $E$ is a $\vee$-closed subset of $E(R)$ consisting of commuting idempotents such that for every $r \in R$ there exists some $f \in E$ such that $f r = r f = r$ (cf. \cite{abrams1983morita} and \cite[Def. 21]{nystedt2018unital}). A ring $R$ is said to have \emph{enough idempotents} if there is a set $F$ of pairwise orthogonal, commuting idempotents such that $\bigvee F$ is a set of local units for $R$ (cf. e.g. \cite[Def. 27]{nystedt2018unital}).
Finally, recall that a ring $R$ is called \emph{$s$-unital} if $x \in xR \cap Rx$ for each $x \in R$. This is equivalent to that there, for every positive integer $n$ and elements $s_1, s_2, \dots, s_n \in R$, exists some $f \in R$ satisfying $fs_i = s_i f = s_i$ for all $1 \leq i \leq n$ (see \cite[Thm. 1]{tominaga1976s}). These definitions relate to each other in the following way.


\begin{proposition}(\cite{nystedt2018unital})
The following strict inclusions hold between the classes of rings.
\begin{align*}
\{ \text{unital rings} \} &\subsetneq \{ \text{rings with enough idempotents} \} \\ & \subsetneq \{ \text{rings with sets of local units} \} \\ &\subsetneq \{ \text{s-unital rings} \}.
\end{align*}\label{prop:local_unital_rings_inclusions}
\end{proposition}
\vspace{-1.5em}
Before defining Nystedt-Öinert-Pinedo graded rings, we recall the following:

\begin{definition}(\cite[Def. 4.5]{clark2018generalized})
A $G$-graded ring $S$ is called \emph{symmetrically graded} if, 
\begin{equation*}
S_g S_{g^{-1}} S_g  = S_g, \qquad \forall g \in G.
\end{equation*}
\end{definition}

Consider a $G$-graded ring $S=\bigoplus_{g \in G}S_g$ where the principal component is $S_e$. Note that $S_g S_{g^{-1}} \subseteq S_e$ is an $S_e$-ideal for each $g \in G$. The classes of Nystedt-Öinert-Pinedo graded rings correspond to symmetrically $G$-graded rings with local unit properties on the rings $S_g S_{g^{-1}}$ according to the following:

\begin{figure}[h]
\begin{tabular}{ c | c  }
class & $S_g S_{g^{-1}}$ \\
\hline
epsilon-strongly (see \cite{nystedt2016epsilon}) & unital ring \\
virtually epsilon-strongly & ring with enough idempotents \\
essentially epsilon-strongly & ring with sets of local units \\
nearly epsilon-strongly (see \cite{nystedt2017epsilon}) & $s$-unital ring 
\end{tabular}
\caption{Nystedt-Öinert-Pinedo classes of graded rings}
\label{fig:1}
\end{figure}


We also explicitly write down these crucial definitions.
\begin{definition}
Let $S=\bigoplus_{g \in G} S_g$ be a symmetrically $G$-graded ring.
\begin{enumerate}[(a)]
\begin{item}
If $S_g S_{g^{-1}}$ is a unital ring for each $g \in G$, then $S$ is called \emph{epsilon-strongly graded} (cf.  \cite[Prop. 7]{nystedt2016epsilon}).
\end{item}
\begin{item}
If $S_g S_{g^{-1}}$ is a ring with enough idempotents for each $g \in G$, then $S$ is called \emph{virtually epsilon-strongly graded}.
\end{item}
\begin{item}
If $S_g S_{g^{-1}}$ is a ring with a set of local units for each $g \in G$, then $S$ is called \emph{essentially epsilon-strongly graded}.
\end{item}
\begin{item}
If $S_g S_{g^{-1}}$ is an $s$-unital ring for each $g \in G$, then $S$ is called \emph{nearly epsilon-strongly graded} (cf. \cite[Prop. 10]{nystedt2017epsilon}).
\end{item}
\end{enumerate}
\label{def:nystedt-gradings}
\end{definition}

The notion of an essentially epsilon-strong grading will be central to our study of the induced quotient group gradings of epsilon-strongly graded rings (see Theorem \ref{thm:main}). Moreover, virtually epsilon-strong gradings relate to Leavitt path algebras (see Proposition \ref{prop:lpa_is_virtually}). 

\begin{remark}
We make some remarks concerning Definition \ref{def:nystedt-gradings}.
\begin{enumerate}[(a)]
\begin{item}
By Proposition \ref{prop:local_unital_rings_inclusions}, we have the following strict inclusions between the classes of Nystedt-Öinert-Pinedo graded rings:
\begin{align*}
\{ \text{epsilon-strongly graded rings} \} & \subsetneq \{ \text{virtually epsilon-strongly graded rings} \} \\ & \subsetneq \{ \text{essentially epsilon-strongly graded rings} \} \\ &\subsetneq \{ \text{nearly epsilon-strongly graded rings} \} .
\end{align*}
\end{item}
\begin{item}
Let $R$ be a ring that has a set of local unit but does not have enough idempotents (see \cite[Expl. 28]{nystedt2018unital}). Then $R$ is graded by the trivial group by putting $S_e:=R$. Note that $R R = R$ is a ring with a set of local units. Hence, $R$ is trivially essentially epsilon-strongly graded. However, since $R$ does not have enough idempotents, this grading cannot be virtually epsilon-strong. In fact, this is our only example distinguishing essentially epsilon-strong gradings from virtually epsilon-strong gradings (see Remark \ref{rem:thm_rem}).
\end{item}
\end{enumerate}
\label{rem:inclusions}
\end{remark}


We recall the following characterization, which will be used implicitly in the rest of this paper. Note that Proposition \ref{prop:epsilon_char}(a) implies that the definition of epsilon-strong given in the introduction is equivalent to Definition \ref{def:nystedt-gradings}(a).
\begin{proposition}(\cite[Prop. 7, Prop. 10]{nystedt2017epsilon})
Let $S=\bigoplus_{g \in G} S_g$ be a $G$-graded ring. The following assertions hold:
\begin{enumerate}[(a)]
\begin{item}
$S$ is epsilon-strongly graded if and only if, for each $g \in G$, there exist $\epsilon_g \in S_g S_{g^{-1}}$ and $\epsilon_{g^{-1}} \in S_{g^{-1}} S_g$ such that $\epsilon_g s = s = s \epsilon_{g^{-1}}$ for every $s \in S_g$;
\end{item}
\begin{item}
$S$ is nearly epsilon-strongly graded if and only if, for each $g \in G$ and $s \in S_g$, there exist $\epsilon_g(s) \in S_g S_{g^{-1}}$ and $\epsilon_g'(s) \in S_{g^{-1}} S_g$ such that $\epsilon_g(s) s = s = s \epsilon_g'(s).$
\end{item}
\end{enumerate}
\label{prop:epsilon_char}
\end{proposition}
For unital rings, the class of strongly graded rings is a subclass of epsilon-strongly graded rings. However, for general non-unital rings, this is not the case.

\begin{proposition}
A unital strongly $G$-graded ring $S=\bigoplus_{g \in G}S_g$ is epsilon-strongly $G$-graded.
\end{proposition}
\begin{proof}
Note that $S_g S_{g^{-1}} =S_e$ is a unital ring for every $g \in G$.
\end{proof}

\begin{example}
(Example of a strongly graded ring that is not epsilon-strongly graded)

Let $R$ be an idempotent ring without multiplicative identity and consider $R$ graded by the trivial group by putting $S_e := R$. Since $S_e S_e = R R = R = S_e$, the grading is strong. On the other hand, $S_e S_e =R$ is not unital. Hence, by definition, the grading cannot be epsilon-strong.
\end{example}
In fact, it turns out that every epsilon-strongly $G$-graded ring is unital. 
For this purpose, we recall that if $R$ is a ring with a left multiplicative identity element $\epsilon$ and a right multiplicative identity element $\epsilon'$, then $\epsilon=\epsilon'$ is a multiplicative identity element of $R$. 
\begin{proposition}
If $S$ is an epsilon-strongly $G$-graded ring, then $S$ is unital with multiplicative identity element $\epsilon_e$. In that case, $S_e$ is a unital subring of $S$.
\label{prop:epsilon_unital}
\end{proposition}
\begin{proof}
Let $R$ denote the principal component of $S$. By Proposition \ref{prop:epsilon_char}, there are two, a priori distinct, elements $\epsilon_e, \epsilon_e' \in S_e^2 \subseteq R$ such that $\epsilon_e r = r = r \epsilon_e'$ for any $r \in R$. This means that $\epsilon_e = \epsilon_e'$ is a multiplicative identity element of $R$. Take an arbitrary element $g \in G$. There are $\epsilon_g \in S_g S_{g^{-1}} \subseteq R$ and $\epsilon_g' \in S_{g^{-1}} S_g \subseteq R$ such that $\epsilon_g s_g = s_g = s_g \epsilon_g'$ for all $s_g \in S_g$. Fix an arbitrary $s_g \in S_g$. Then, 
\begin{equation*}
\epsilon_e s_g = \epsilon_e  (\epsilon_g s_g) = (\epsilon_e \epsilon_g) s_g = \epsilon_g s_g = s_g,
\end{equation*} 
and similarly, $s_g \epsilon_e = s_g$. Since a general element $s \in S$ is a finite sum $s=\sum s_g$ of elements $s_g \in S_g$, it follows that $\epsilon_e$ is a multiplicative identity element of $S$. 
\end{proof}

\begin{remark}
Note that by Proposition \ref{prop:epsilon_unital}, only unital rings admit an epsilon-strong grading. However, there are a lot of virtually epsilon-strongly graded rings which are not unital. In the next section we will give an example of such a ring (Example \ref{ex:derp22}).
\end{remark}

For the remainder of this section, we briefly consider gradings of the factor ring $S/I$. If $S$ is $G$-graded, then $S/I$ inherits a $G$-grading for certain ideals $I$. More precisely, let $G$ be an arbitrary group and let $S=\bigoplus_{g \in G} S_g$ be a $G$-graded ring. Recall that an ideal $I$ of $S$ is called \emph{homogeneous} (or graded) if $I = \bigoplus_{g \in G} (I \cap S_g)$. If $I$ is a homogeneous ideal, then the factor ring is naturally $G$-graded by,
\begin{equation}
S/I = \bigoplus_{g \in G} S_g / (I \cap S_g) = \bigoplus_{g \in G} (S_g + I)/I.
\label{eq:factor_grading}
\end{equation}
We will show that if $S$ is epsilon-strongly $G$-graded, then $S/I$ is epsilon-strongly $G$-graded.

\begin{lemma}
Let $\phi \colon A \to B$ be a $G$-graded epimorphism of $G$-graded rings $A$ and $B$. If $A$ is epsilon-strongly $G$-graded, then $B$ is epsilon-strongly $G$-graded.
\label{lem:graded_epi}
\end{lemma}
\begin{proof}
Suppose $A = \bigoplus_{g \in G} A_g$ and $B= \bigoplus_{g \in G} B_g$. For every $g \in G$, let $\epsilon_g$ be the multiplicative identity element of $A_g A_{g^{-1}}$. Using that $A$ is epsilon-strongly $G$-graded, we have that $A_g A_{g^{-1}} A_g = A_g$ and $A_g A_{g^{-1}} = \epsilon_g A_e$ for every $g \in G$. Since $\phi$ is a $G$-graded epimorphism, we have that $\phi(A_g)=B_g$ for every $g \in G$. Applying $\phi$ to both equations we get $B_g B_{g^{-1}} B_g = B_g$ and $B_g B_{g^{-1}} = \phi(\epsilon_g) B_e$ for every $g \in G$. Hence, $B$ is epsilon-strongly $G$-graded. 
\end{proof}

\begin{proposition}
Let $S$ be an epsilon-strongly $G$-graded ring. If $I$ is a homogeneous ideal of $S$, then the natural $G$-grading of $S/I$ is epsilon-strong.
\end{proposition}
\begin{proof}
Follows by Lemma \ref{lem:graded_epi} since the natural epimorphism  $\pi \colon S \to S/I$ is $G$-graded.
\end{proof}

\section{Leavitt path algebras}
\label{sec:3}

In this section, we will show that the Leavitt path algebra associated to any directed graph is virtually epsilon-strongly graded. Let $R$ be an arbitrary unital ring and let $E=(E^0, E^1, s, r)$ be a directed graph. Here, $E^0$ denotes the vertex set, $E^1$ denotes the set of edges and the maps $s \colon E^1 \to E^0, r \colon E^1 \to E^0$ specify the \emph{source} and the \emph{range}, respectively, of each edge $f \in E^1$. The Leavitt path algebra $L_R(E)$ of $E$ with coefficients in $R$ is an algebraic analogue of the graph $C^*$-algebra associated to $E$. For more details about Leavitt path algebras, we refer the reader to the monograph by Abrams, Ara, and Siles Molina \cite{abrams2017leavitt}. In the case of $R$ being a field, Leavitt path algebras were first considered by Ara, Moreno and Pardo \cite{ara2007nonstable} and Abrams and Aranda Pino in \cite{abrams2005leavitt}. Later, Tomforde \cite{tomforde2011leavitt} considered Leavitt path algebras with coefficients in a commutative ring. We follow Hazrat \cite{hazrat2013graded} and let $R$ be a general (possibly non-commutative) unital ring.
 
\begin{definition}
For a directed graph $E= (E^0, E^1, s, r)$ and a unital ring $R$, the \textit{Leavitt path algebra with coefficients in $R$} is the $R$-algebra $L_R(E)$ generated by the symbols $\{ v \mid v \in E^0 \}$, $\{ f \mid f \in E^1 \}$ and $\{ f^* \mid f \in E^1 \}$ subject to the following relations:
\begin{enumerate}[(a)]
\begin{item}
$ u v = \delta_{u,v} u$ for all $u, v \in E^0$,
\end{item}
\begin{item}
$s(f) f = f r(f)=f$ and $r(f)f^* = f^*s(f)=f^*$  for all $f \in E^1$,
\end{item}
\begin{item}
$f^* f' = \delta_{f, f'} r(f)$ for all $f, f' \in E^1$,
\end{item}
\begin{item}
$\sum_{f \in E^1, s(f)=v}f f^* =v $ for all $v \in E^0$ for which $s^{-1}(v)$ is non-empty and finite. 
\end{item}
\end{enumerate}
We let $R$ commute with the generators. 
\label{def:lpa}
\end{definition}

A \emph{path} is a sequence of edges $\alpha=f_1 f_2 \dots f_n$ such that $r(f_i)=s(f_{i+1})$ for $1 \leq i \leq n-1$. We write $s(\alpha) = s(f_1)$ and $r(\alpha)=r(f_n)$. Using the relations in Definition \ref{def:lpa}, it can be shown that a general element of $L_R(E)$ has the form of a finite sum $\sum r_i \alpha_i \beta_i^*$ where $r_i \in R$, $\alpha_i$ and $\beta_i$ are paths such that $r(\alpha_i)=s(\beta_i^*)=r(\beta_i)$ for every $i$. There is an anti-graded involution on $L_R(E)$ defined by $f \mapsto f^*$ for every $f \in E^1$. The image of a path $\alpha = f_1 f_2 \dots f_n$ under the involution is $\alpha^* = f_n^* f_{n-1}^* \dots f_1^*$. Note that for any elements $\alpha, \beta \in L_R(E)$ we have that $(\alpha \beta)^* = \beta^* \alpha^*$.

Next, we recall (see e.g. \cite[Sect. 4]{nystedt2017epsilon}) a process to assign a $G$-grading to $L_R(E)$ for an arbitrary group $G$. Let $F_R(E) = R \langle v, f, f^* \mid v \in E^0, f \in E^1 \rangle$ denote the free $R$-algebra generated by all symbols of the form $v, f, f^*$. Put $\text{deg}(v) = e$ for each $v \in E^0$. For every $f \in E^1$, choose a $g \in G$ and put $\text{deg}(f) = g$ and $\text{deg}(f^*)=g^{-1}$ This extends to a $G$-grading of $F_R(E)$ in the obvious way. Next, let $J$ be the ideal generated by the relations (a)-(d) in Definition \ref{def:lpa}. It is easy to check that $J$ is homogeneous and thus the factor algebra $L_R(E)=F_R(E)/J$ is $G$-graded by the factor $G$-grading (see (\ref{eq:factor_grading})). This $G$-grading is called a \emph{standard $G$-grading} of $G$. Note that this construction depends on which elements $g \in G$ we assign to the generators. In the special case of $G = \mathbb{Z}$, the natural choice is to put $\deg(f)=1$ and $\deg(f^*)=-1$ for all $f \in E^1$. In this special case, $\text{deg}(\alpha)=\text{len}(\alpha)$ for any real path $\alpha$. The resulting grading is called the \emph{canonical $\mathbb{Z}$-grading} of $L_R(E)$. For a general standard $G$-grading, the homogeneous component of degree $h \in G$ can be expressed as,  
\begin{equation*}
 (L_R(E))_h = \text{Span}_R \{ \alpha \beta^* \mid \text{deg}(\alpha) \text{deg}(\beta)^{-1} = h, r(\alpha)=r(\beta) \}.
\label{eq:2}
\end{equation*}

The canonical $\mathbb{Z}$-grading was investigated by Hazrat \cite{hazrat2013graded}. Among other results, he gave a criterion on the finite graph $E$ for the Leavitt path algebra $L_R(E)$ to be strongly $\mathbb{Z}$-graded (see \cite[Thm. 3.11]{hazrat2013graded}). Continuing the investigation of the group graded structure of Leavitt path algebras, Nystedt and Öinert introduced the class of nearly epsilon-strongly graded rings and proved the following.

\begin{proposition}(\cite[Thm. 28]{nystedt2017epsilon}) Let $G$ be an arbitrary group, let
$R$ be a unital ring and let $E$ be a directed graph. Then any standard $G$-grading of $L_R(E)$ is nearly epsilon-strong. Note that, in particular, $L_R(E)$ is symmetrically $G$-graded.
\label{prop:lpa_nearly}
\end{proposition}

Having introduced the notion of virtually epsilon-strongly graded rings, the author of the present paper realized that Proposition \ref{prop:lpa_nearly} could be made more precise. We recall that $L_R(E)$ is a ring with enough idempotents (see e.g. \cite[Lem. 1.2.12(v)]{abrams2017leavitt}). In line with this property, it turns out that any standard $G$-grading of $L_R(E)$ is virtually epsilon-strong. 

\begin{proposition} Let $G$ be an arbitrary group.
If $R$ is a unital ring and $E$ is any graph, then every standard $G$-grading of $L_R(E)$ is virtually epsilon-strong.
\label{prop:lpa_is_virtually}
\end{proposition}
\begin{proof}
To save space we denote $(L_R(E))_g$ by $S_g$. Since the standard $G$-grading is symmetric by Proposition \ref{prop:lpa_nearly}, it is enough to show that $S_g S_{g^{-1}}$ is a ring with enough idempotents for each $g \in G$.

Take $g \in G$. We will show that $S_g S_{g^{-1}}$ has enough idempotents by constructing a set $M$ of pairwise orthogonal, commuting idempotents of $S_g S_{g^{-1}}$ such that $E_g = \bigvee M$ is a set of local units for $S_g S_{g^{-1}}$. Let $A = \{ \alpha \mid \alpha \beta^* \in S_g S_{g^{-1}} \}$ and define a partial order of $A$ by letting $\alpha \leq \beta$ if and only if $\alpha$ is an initial subpath of $\beta$. Next, let, $$ M = \{ \alpha \alpha^* \mid \alpha \text{ minimal in } (A, \leq) \}.$$ By construction, if $\alpha \alpha^* \in M$ there exists some $\beta$ such that $\alpha \beta^* \in S_g S_{g^{-1}}$. Then, $\beta \alpha^* \in S_{g^{-1}}S_g$. Hence, $\alpha \alpha^* = (\alpha \beta^*) (\beta \alpha^*) \in (S_g S_{g^{-1}}) (S_{g^{-1}}S_g) \subseteq (S_g S_{g^{-1}}) S_e =  S_g S_{g^{-1}}$. Thus, $M \subseteq S_g S_{g^{-1}}$.

Recall (see \cite[Lem. 1.2.12]{abrams2017leavitt}) that the following equation holds for any paths $\gamma, \delta, \lambda, \rho$:
\begin{equation}
(\gamma \delta^*)(\lambda \rho^*) = \begin{cases}
\gamma \kappa \rho^* & \text{if } \lambda = \delta \kappa \text{ for some path } \kappa \\
\gamma \sigma^* \rho^* & \text{if } \delta = \lambda \sigma \text{ for some path } \sigma \\
0 & \text{otherwise}
\end{cases}
\label{eq:a1}
\end{equation}
In particular, for any paths $\alpha, \beta$:
\begin{equation}
(\alpha \alpha^*)(\beta \beta^*) = (\beta \beta^*)(\alpha \alpha^*) = \begin{cases}
\alpha \alpha^* & \alpha = \beta \\
\beta \beta^* & \alpha \leq \beta \\
\alpha \alpha^* & \beta \leq \alpha \\
0 & \text{otherwise}
\end{cases}
\label{eq:1}
\end{equation}
It follows from (\ref{eq:1}) that the set $M$ consists of pairwise orthogonal, commuting idempotents. 

Note that for each $\alpha \beta^* \in S_g S_{g^{-1}}$ there exists a unique element $\delta \delta^* \in M$ such that $\delta \leq \alpha$. Hence, we can define a function by $\mu(\alpha \beta^*)= \delta \delta^*$. Since $\alpha = \delta \delta'$ for some path $\delta'$, it follows by (\ref{eq:a1}) that $\mu(\alpha \beta^*) \alpha \beta^* = (\delta \delta^*)( \alpha \beta^*)=\delta \delta' \beta^* = \alpha \beta^*$ for any $\alpha \beta^* \in S_g S_{g^{-1}}$. On the other hand, let $\alpha_1 \beta_1^*, \alpha_2 \beta_2^* \in S_g S_{g^{-1}}$ such that, 
\begin{equation}
\delta_1 \delta_1^* = \mu(\alpha_1 \beta_1^*) \ne \mu(\alpha_2 \beta_2^*) = \delta_2 \delta_2^*,
\label{eq:a2}
\end{equation} for some paths $\delta_1, \delta_2$ satisfying $\delta_1 \leq \alpha_1$ and $\delta_2 \leq \alpha_2$. Note that $\delta_1 \nleq \alpha_2$ and $\delta_2 \not \leq \alpha_1$ as otherwise (\ref{eq:a2}) would not hold. Moreover, it follows by (\ref{eq:a2}) and the definition of $M$ that $\alpha_2   \nleq \delta_1 $ and $\alpha_1 \not \leq  \delta_2 $. Thus, by (\ref{eq:a1}), we have that $\mu(\alpha_1 \beta_1^*)  \alpha_2 \beta_2^* = (\delta_1 \delta_1^*) ( \alpha_2 \beta_2^*) = 0$ and similarly $ \mu(\alpha_2 \beta_2^*) \alpha_1 \beta_1^* = (\delta_2 \delta_2^*) (\alpha_1 \beta_1^*)  = 0$. 

Consider an arbitrary $s = \sum_{i \in I} r_i \alpha_i \beta_i^* \in S_g S_{g^{-1}}$ for some finite index set $I$. Let $J \subseteq I$ the subset of elements $\alpha_i \beta_i^*$ with unique images under the map $\mu$, i.e. for all $i \in I$ there is some $j \in J$ such that $\mu(\alpha_i \beta_i^*) = \mu(\alpha_j \beta_j^*)$. Put $e = \bigvee_{j \in J} \mu(\alpha_j \beta_j^*) = \sum_{j \in J} \mu(\alpha_j \beta_j^*)$. Then,
\begin{align*}
e s &= \Big ( \sum_{j \in J} \mu(\alpha_i \beta_i^*) \Big ) \Big ( \sum_{i \in I} r_i \alpha_i \beta_i^* \Big ) = \sum_{i \in I, j \in J} \mu(\alpha_j \beta_j^*) r_i \alpha_i \beta_i^* \\ &= \sum_{i \in I} \mu(\alpha_i \beta_i^*) r_i \alpha_i \beta_i^* =  \sum_{i \in I} r_i \mu(\alpha_i \beta_i^*)  \alpha_i \beta_i^* = \sum_{i \in I} r_i \alpha_i \beta_i^* = s.
\end{align*}
Note that $s^* = \sum_{i \in I} (\alpha_i \beta_i^*)^* = \sum_{i \in I}\beta_i \alpha_i^* \in S_g S_{g^{-1}}$. Hence, by the above argument there exists some element $f \in \bigvee M \subseteq S_g S_{g^{-1}}$ such that $f s^* = s^*$. But by construction $f= \sum \alpha \alpha^*$ for some finite sum, implying that $f^* = \sum (\alpha \alpha^*)^* = \sum \alpha \alpha^* = f$. Thus, $s = (s^*)^* = (f s^*)^* = (s^*)^* f^* = s f^* = s f$.

Finally, note that $e \vee f \in \bigvee M$ and $(e \vee f) s = s (e \vee f) = s$. Hence, $E_g = \bigvee M$ is a set of local units for $S_g S_{g^{-1}}$. 
\end{proof}

We end this section with two examples that distinguish the class of virtually epsilon-strongly graded rings from strongly graded rings and epsilon-strongly graded rings.

\begin{example}
(Virtually epsilon-strongly graded but not strongly graded)
Let $R$ be a unital ring and let $E$ be the finite graph in Figure \ref{fig:graph1}. Note that $L_R(E)$ is not strongly $\mathbb{Z}$-graded since $v_2$ is a sink (see \cite[Thm. 3.15]{hazrat2013graded}). However, $L_R(E)$ is epsilon-strongly $\mathbb{Z}$-graded since $E$ is finite (see \cite[Thm. 24]{nystedt2017epsilon}). Recall that epsilon-strongly graded implies virtually epsilon-strongly graded (see Remark \ref{rem:inclusions}(a)). Thus, $L_R(E)$ is virtually epsilon-strongly $\mathbb{Z}$-graded but not strongly $\mathbb{Z}$-graded.

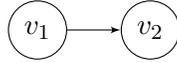
\begin{figure}[h!]
\begin{tikzpicture}[scale=1.5]
\tikzset{vertex/.style = {shape=circle,draw,minimum size=1.5em}}
\tikzset{edge/.style = {->,> = latex'}}

\node[vertex] (a) at  (0,0) {$v_1$};
\node[vertex] (b) at  (1,0) {$v_2$};

\draw[edge] (a) to (b);
\end{tikzpicture}
\caption{Finite graph with a sink}
\label{fig:graph1}
\end{figure}
\end{example}

\begin{example}(Virtually epsilon-strongly graded but not epsilon-strongly graded)
Let $R$ be a unital ring and let $E$ be a graph consisting of infinitely many disjoint vertices, i.e. $E^0 = \{ v_n \mid n \geq 0 \}$ and $E^1 = \emptyset$ (see Figure \ref{fig:graph2}). Consider $L_R(E)$ with the canonical $\mathbb{Z}$-grading. Since $E^0$ is infinite, $(L_R(E))_0$ does not admit a multiplicative identity element. Hence, by Proposition \ref{prop:epsilon_unital}, $L_R(E)$ cannot be epsilon-strongly graded. Furthermore, for any integer $k$ such that $|k| > 0$, $(L_R(E))_k = \{ 0 \}$. Hence, $(L_R(E))_k (L_R(E))_{-k} = \{0 \} \ne (L_R(E))_0$ for $|k| > 0$ implying that $L_R(E)$ is not strongly $\mathbb{Z}$-graded. On the other hand, $L_R(E)$ is virtually epsilon-strongly $\mathbb{Z}$-graded by Proposition \ref{prop:lpa_is_virtually}.

\begin{figure}[h!]
\begin{tikzpicture}[scale=1.5]
\tikzset{vertex/.style = {shape=circle,draw,minimum size=1.5em}}
\tikzset{edge/.style = {->,> = latex'}}

\node[vertex] (a) at  (0,0) {$v_1$};
\node[vertex] (b) at  (1,0) {$v_2$};
\node[vertex] (c) at  (2,0) {$v_3$};
\node[vertex] (d) at  (3,0) {$v_4$};
\node[vertex] (e) at  (4,0) {$v_5$};
\node[vertex] (f) at  (5,0) {$v_6$};
\node[vertex] (g) at  (6,0) {$v_7$};

\node[vertex, draw=white] (h) at (7,0) {$...$};

\end{tikzpicture}
\caption{Discrete infinite graph}
\label{fig:graph2}
\end{figure}
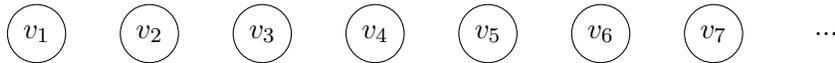

\label{ex:derp22}
\end{example}

\section{Induced quotient group gradings of Nystedt-Öinert-Pinedo rings}
\label{sec:4}

In this section, we will derive our main results about the induced quotient group gradings of epsilon-strongly graded rings (see Theorem \ref{thm:main} and Theorem \ref{thm:main1}). This involves a systematic study of induced quotient group gradings of Nystedt-Öinert-Pinedo rings. Recall that for idempotents $e, f \in E(R)$ we write $e \leq f$ if and only if $e = ef = fe$, i.e. $e$ absorbs $f$. We will think about $e \leq f$ as expressing that $f$ can be used as a ``local unit'' in place of $e$. More precisely, we have the following.

\begin{lemma}
Let $e,f \in E(R)$ be idempotents of $R$. Then, $e \leq f$ is equivalent to the following: For any $x \in R$,  
\begin{enumerate}[(a)]
\begin{item}
$x=ex \implies x = f x$, and,  
\end{item}
\begin{item}
$x = x e \implies x = x f$. 
\end{item}
\end{enumerate}
\label{lem:idem}
\end{lemma}
%


\begin{proposition}
Let $R$ be a ring and let $E$ be a set of local units for $R$. Then $R$ is a unital ring if and only if there exists some $e' \in E(R)$ such that $e \leq e'$ for all $e \in E$. If such an element exists, then it is the multiplicative identity element of $R$.
\label{prop:greatest}
\end{proposition}
\begin{proof}
Assume that $e'\in E(R)$ satisfies $e \leq e'$ for all $e \in E$. Let $x \in R$ be any element. Then, there is some $e_x \in E$ such that $x=e_x x = x e_x$. But by assumption, $e_x \leq e'$. By Lemma \ref{lem:idem}, $x=e'x=xe'$. Hence $e'$ is the multiplicative identity element of $R$.

Conversely, assume that $R$ is unital with multiplicative identity $1$. Then, $e \leq 1$ for all $e \in E$. 
\end{proof}

\begin{example}
Consider $R_1 := \bigoplus_{\mathbb{Z}} \mathbb{Q}$ and $R_2 := \prod_{\mathbb{Z}} \mathbb{Q}$. For $i \in \mathbb{Z}$, let $\delta_i \colon \mathbb{Z} \to \mathbb{Q}$ be the function such that for each integer $j$, $\delta_i(j) = \delta_{i,j}$ where $\delta_{i,j}$ is the Kronecker delta. Since $\delta_i$ has finite support, $\delta_i \in R_1 \subset R_2$ for each $i \in \mathbb{Z}$. It is easy to see that the set $E=\bigvee \{ \delta_i \mid i \in \mathbb{Z} \}$ is a set of local units for both $R_1$ and $R_2$. Note that $R_2$ is unital with multiplicative identity $1_{R_2} = \bigvee_{i \in \mathbb{Z}} \delta_i$ while $R_1$ is not unital.
\end{example}

The following characterization becomes very useful when $E$ is a finite set. 

\begin{corollary}
Let $R$ be a ring with a set of local units $E$. If $E$ is a finite set, then $(E,\leq)$ has a greatest element and thus $R$ is a unital ring.
\label{cor:finite} 
\end{corollary}
\begin{proof}
Assume that $e_1, \dots, e_n$ are the elements of $E$. Then $e_1 \vee \dots \vee e_n \in E$ exists since it is a finite join. Furthermore, $e_1 \vee \dots \vee e_n \in R$ is the greatest element of $E$ with respect to the ordering $\leq$. Thus, $R$ is unital by Proposition \ref{prop:greatest}. 
\end{proof}

Before considering induced quotient group gradings of Nystedt-Öinert-Pinedo rings, we show that the functor $U_{G/N}$ restricts to the category of symmetrically $G$-graded rings.

\begin{proposition}
Let $G$ be an arbitrary group and let $N$ be a normal subgroup of $G$. If $S$ is symmetrically $G$-graded, then the induced $G/N$-grading is symmetric.
\label{prop:symmetric}
\end{proposition}
\begin{proof}
We need to show that for any class $C \in G/N$, we have that $S_C S_{C^{-1}} S_C = S_C$. The inclusion $S_C S_{C^{-1}} S_C \subseteq S_{C C^{-1} C } = S_C$ holds for any $C \in G/N$. Let $[g]=gN \in G/N$ denote the coset of $g \in G$. It remains to prove that $S_{[g]} \subseteq S_{[g]} S_{[g]^{-1}} S_{[g]}$ for all $g \in G$. But $S_{[g]} = \bigoplus_{n \in N} S_{gn}$, hence it is enough to show ${S_{gn} \subseteq S_{[g]} S_{[g]^{-1}} S_{[g]}}$ for all $n \in N$ and $g \in G$. But, since $S$ is symmetrically $G$-graded, it holds that ${S_{gn} \subseteq S_{gn} S_{(gn)^{-1}} S_{gn}}$. Furthermore, $(gn)^{-1} = n^{-1} g^{-1} = g^{-1} n' $ for some $n' \in N$ since $N$ is normal.  Thus, $S_{gn} \subseteq S_{[g]} S_{[g]^{-1}} S_{[g]}$.   
\end{proof}

Note that Proposition \ref{prop:symmetric} implies that the induced quotient group grading of any Nystedt-Öinert-Pinedo ring is symmetric. Hence, we will focus on deciding which of the local unit properties of Definition \ref{def:nystedt-gradings} are preserved. 

\smallskip

Before considering nearly epsilon-strongly graded rings, we recall the following useful but somewhat obscure result. For the convenience of the reader, we include a proof.

\begin{lemma}
(\cite[Thm. 1]{tominaga1976s}) Let $T$ be a ring and let $M$ be a left (right) $T$-module. Take a finite subset $X \subseteq M$ and assume that for each $x \in X$ there is some $e_x \in T$ such that $e_x x = x$ ($x e_x = x$). Then, there is some $e \in T$ such that $ex =x$ ($xe=x$) for all $x \in X$.
\label{lem:tominaga}
\end{lemma}
\begin{proof}
We only prove the left case as the right case is treated analogously. The proof goes by induction on the size of the set $X = \{ x_1, x_2, \dots, x_n \}$.

The base case $n=1$ is clear. 

Assume that the lemma holds for $n=k$ for some $k > 0$ and consider the subset, $$\{ x_1, x_2, \dots x_k, x_{k+1} \} \subseteq X.$$ Then, for all $1 \leq i \leq k+1$, there is some $e_i \in T$ such that $e_i x_i = x_i$. For $1 \leq i \leq k$, let $v_i = x_i - e_{k+1} x_i$. By the induction hypothesis, there is some $e' \in T$ such that $e' v_i = v_i$ for $1 \leq i \leq k$. Put $e = e' + e_{k+1} - e' e_{k+1}$. It is clear that $e \in T$. Moreover, 
\begin{equation*}
e x_{k+1} =  e' x_{k+1} + e_{k+1} x_{k+1} - e' e_{k+1} x_{k+1} = e' x_{k+1} + x_{k+1} - e' x_{k+1} = x_{k+1},
\end{equation*}
and, for $1 \leq i \leq k$, 
\begin{align*}
e x_{i} &= e' x_{i} + e_{k+1} x_{i} - e' e_{k+1} x_{i} = e' (x_{i} - e_{k+1} x_i ) + e_{k+1} x_i = e' v_i + e_{k+1} x_i \\ &= v_i +e_{k+1}x_i = x_i.
\end{align*}
Hence, the lemma follows by the induction principle.
\end{proof}

Throughout the rest of this section, let $N$ be a normal subgroup of $G$ and let $S = \bigoplus_{g \in G} S_g$ be a nearly epsilon-strongly $G$-graded ring.

\begin{lemma}
Consider a fixed class $C \in G/N$. For any positive integer $n$ and elements $s_{g_1}, s_{g_2}, \dots, s_{g_n}$ such that $s_{g_i} \in S_{g_i}$ and $g_i \in C$ for $1 \leq i \leq n$, there exists some $e \in S_C S_{C^{-1}}$ ($e' \in S_{C^{-1}} S_C$) such that $e s_{g_i} = s_{g_i}$ ($s_{g_i} e' = s_{g_i}$) for all $1 \leq i \leq n$. 
\label{lem:inv_exist}
\end{lemma}
\begin{proof}
We only prove the left case as the right case is treated analogously. Let $T= S_C S_{C^{-1}}$, $M=S_C$ and $X= \{ s_{g_1}, s_{g_2}, \dots, s_{g_n} \}$. Take an arbitrary integer $1 \leq i \leq n$. Since $S$ is nearly epsilon-strongly graded, there is some $e \in S_{g_i} S_{(g_i)^{-1}} \subseteq T$ such that $e s_{g_i} = s_{g_i}$. The statement now follows from Lemma \ref{lem:tominaga}.
\end{proof}
%
%

We can now prove that the functor $U_{G/N}$ restricts to the category of nearly epsilon-strongly $G$-graded rings.

\begin{proposition} 
Let $G$ be an arbitrary group and let $N$ be a normal subgroup of $G$. If $S = \bigoplus_{g \in G} S_g$ is a nearly epsilon-strongly $G$-graded ring, then the induced $G/N$-grading is nearly epsilon-strong.
\label{prop:nearly_induced}
\end{proposition}
\begin{proof}
Take an arbitrary $C \in G/N$ and $s \in S_C$. We need to show that there exist some $\epsilon_C(s) \in S_C S_{C^{-1}}$ and $\epsilon_C(s)' \in S_{C^{-1}} S_C$ such that $\epsilon_C(s)s=s\epsilon_C(s)'=s$. Note that $s$ can be written as $s = s_{g_1} + s_{g_2} + \dots + s_{g_n}$ for some positive integer $n$ and elements $s_{g_i} \in S_{g_i}$ such that $g_i \in C$ for all $1 \leq i \leq n$. Hence, by Lemma \ref{lem:inv_exist}, there exist $e \in S_C S_{C^{-1}}$ and $e' \in S_{C^{-1}} S_C$ (depending on $C$ and $s$) such that $e s = s e' = s$. 
\end{proof}

We now consider essentially epsilon-strongly graded rings. 

\begin{lemma}
Let $S=\bigoplus_{g \in G}S_g$ be essentially epsilon-strongly $G$-graded where for each $g \in G$, $E_g$ is a set of local units for the ring $S_gS_{g^{-1}}$. Then for any $g \in G$ and any $s \in S_g$ there exist some $e \in E_g$ and $e' \in E_{g^{-1}}$ such that $es = s = s e'$. 
\label{lem:almost_module}
\end{lemma}
\begin{proof}
Take an arbitrary $g \in G$. Since $S$ is symmetrically graded, $S_g = S_g S_{g^{-1}} S_g= (S_g S_{g^{-1}}) S_g$. Hence, for any $s \in S_g$, we can write $s = r s'$ for some $r \in S_g S_{g^{-1}}$ and $s' \in S_g$. But since $E_g$ is a set of local units of $S_g S_{g^{-1}}$ there exists some $e \in E_g$ such that $e r = r = r e$. This implies that, $$e s = e (r s') =  (e r) s' = rs' = s.$$ The existence of $e' \in E_{g^{-1}}$ is proved similarly.
\end{proof}

Let $N$ be a normal subgroup of $G$ and let $S=\bigoplus_{g \in G} S_g$ be an essentially epsilon-strongly $G$-graded ring. By assumption $S_g S_{g^{-1}}$ has set of local units $E_g$ for each $g \in G$. We shall show that taking joins of these elements will be enough to construct local units for the rings $S_C S_{C^{-1}}$. For a family of sets $\{ E_i \}_{i \in I}$ we let $\bigvee_{i \in I} E_i$ denote the set of finite joins of elements in $\bigcup_{i \in I} E_i$. Note that $\bigvee_{i \in I} E_i$ is the $\vee$-closure of $\bigcup_{i \in I} E_i$.

\begin{proposition}
Consider a fixed class $C \in G/N$. Assume that for any $g, h \in C$ and any $e \in E_g, f \in E_h$ we have that $e \vee f$ exists. Then, $E_C = \bigvee_{g \in C} E_g$ is a set of local units for $S_C S_{C^{-1}}$. 
\label{prop:almost_local_units}
\end{proposition}
\begin{proof}
Let $z \in S_C S_{C^{-1}}$. We shall construct an element $e \in E_C$ (depending on $C$ and $z$) such that $ez=ze=z$. Note that $z =\sum_i x_i y_i$  where the sum is finite and $x_i \in S_C$ and $y_i \in S_{C^{-1}}$ for each $i$. Next, consider a fixed $i$. Then, $x_i$ decomposes uniquely as a finite sum $x_i = \sum s_g$ where $s_g \in S_g$.  Let $\text{Supp}(x_i) := \{ g \in G \mid s_g \ne 0\}$ and note that $\text{Supp}(x_i)$ is a finite subset of $C$.  Similarly, $y_i$ decomposes uniquely as a finite sum $y_i = \sum t_h$ where $t_h \in S_h$. Let $\text{Supp}(y_i) := \{ h \in G \mid t_h \ne 0 \}$ and note that $\text{Supp}(y_i)$ is a finite subset of $C^{-1}$. Hence, the sets,   

\begin{equation}
 A = \bigcup_{i} \text{Supp}(x_i), \quad B = \bigcup_{i} \text{Supp}(y_i),
\label{eq:l27}
\end{equation}

are both finite. By substituting the expressions for $x_i$ and $y_i$ in the definition of $z$ we see that $z$ is a finite sum of elements  $s_g t_h$ with $g \in A$, $h \in B$, $s_g \in S_g$ and $t_h \in S_h$.  

Take $g \in A, h \in B$ and corresponding elements $s_g \in S_g, t_h \in S_h$. By Lemma \ref{lem:almost_module}, there are $\epsilon_g(s_g) \in E_g \subseteq E_C$ and $\epsilon_h'(t_h) \in E_{h^{-1}} \subseteq E_C$ such that $\epsilon_g(s_g) s_g = s_g$ and $ t_h \epsilon_h'(t_h) = t_h$. From Lemma \ref{lem:idem} it follows that, 
\begin{equation*}
s_g t_h = (\epsilon_g(s_g) \vee \epsilon_h'(t_h)) s_g t_h = s_g t_h (\epsilon_g(s_g) \vee \epsilon_h'(t_h)).
\end{equation*}

Now, let
\begin{equation}
e = \bigvee_{g \in A} \epsilon_g(s_g) \vee \bigvee_{h \in B} \epsilon_h'(t_h).
\label{eq:l28}
\end{equation}
Since the sets in (\ref{eq:l27}) are finite, the upper bounds on the right hand side of (\ref{eq:l28}) exist by our assumptions. Furthermore, $e \in E_C$ by construction. Moreover, by (\ref{eq:l28}) and noting that $\epsilon_g(s_g) \vee \epsilon_h'(t_h) \leq e$ for all $g \in A, h \in B$,
\begin{equation*}
e z = e \sum s_g t_h = \sum (e s_g t_h) = \sum s_g t_h = z.
\end{equation*}
Similarly, $ z e = z$. Hence, $E_C$ is a set of local units for $S_C S_{C^{-1}}$.
\end{proof}

We can now finish the essentially epsilon-strongly graded case, with the following conclusion.

\begin{corollary}
Let $N$ be a normal subgroup of $G$ and let $S=\bigoplus_{g \in G}S_g$ be essentially epsilon-strongly $G$-graded where for each $g \in G$, $E_g$ is a set of local units for $S_g S_{g^{-1}}$. For every class $C \in G/N$, assume that for any $g,h \in C$ and any $e \in E_g, f \in E_h$ the join $e \vee f$ exists. Then, the induced $G/N$-grading is essentially epsilon-strong.
\label{cor:essentially_induced}
\end{corollary}
\begin{proof}
By Proposition \ref{prop:symmetric} the induced $G/N$-grading is symmetric. Furthermore, by Proposition \ref{prop:almost_local_units}, it follows that for each $C \in G/N$, the ring $S_C S_{C^{-1}}$ has a set of local units. Hence, the induced $G/N$-grading is essentially epsilon-strong.
\end{proof}

We now consider the virtually epsilon-strongly graded case. It turns out that we will need some further assumptions in order to prove that the induced quotient group grading is virtually epsilon-strong.

\begin{lemma}
Let $S=\bigoplus_{g \in G} S_g$ be virtually epsilon-strongly $G$-graded where for each $g \in G$, $M_g$ is a set of pairwise orthogonal, commuting idempotents such that $E_g = \bigvee M_g$ is a set of local units for $S_g S_{g^{-1}}$. Then, for any $g \in G$ and $s \in S_g$ there exist $m_1, m_2, \dots, m_i \in M_g$ and $m_1', m_2', \dots, m_j' \in M_{g^{-1}}$ such that, 
\begin{equation*}
( m_1 \vee m_2 \vee \dots \vee m_i ) s = s,
\end{equation*}
and,
\begin{equation*}
s ( m_1' \vee m_2' \vee \dots \vee m_j') = s.
\end{equation*}
\end{lemma}
\begin{proof}
Follows from Lemma \ref{lem:almost_module} and the fact that $E_g = \bigvee M_g$ is a set of local units for $S_g S_{g^{-1}}$.
\end{proof}

\begin{proposition}
Let $N$ be a normal subgroup of $G$ and let $S=\bigoplus_{g \in G}S_g$ be virtually epsilon-strongly $G$-graded where for each $g \in G$, $M_g$ is a set of commuting orthogonal idempotents such that $E_g = \bigvee M_g$ is a set of local units for $S_g S_{g^{-1}}$.  
For each class $C \in G/N$, let $A_C := \bigvee_{g \in C} M_g$ and assume that the following statements hold:
\begin{enumerate}[(a)]
\begin{item}
For any $g,h \in C$ and $e \in M_g, f \in M_h$ we have that $ef =fe$;
\end{item}
\begin{item}
The maximal elements of $(A_C, \leq)$ are pairwise orthogonal idempotents;
\end{item}
\begin{item}
Every chain in the poset $(A_C, \leq)$ is finite.
\end{item}
\end{enumerate}
Then, the induced $G/N$-grading is virtually epsilon-strong.
\label{prop:virtually_induced}
\end{proposition}
\begin{proof}
Consider the induced $G/N$-grading and take an arbitrary $C \in G/N$. To establish the proposition we need to show that $S_C S_{C^{-1}}$ is a ring with enough idempotents. Put, 
$A_C = \bigvee_{g \in C} M_g$ and note that $ef=fe \in A_C$ for any $e, f \in A_C$. In other words, $A_C$ is a set of commuting idempotents that is closed with regards to the join operator $\vee$. We want to choose the maximal elements from $A_C$ with respect to the idempotent ordering. More precisely, let,
\begin{equation*}
D_C = \{ e \in A_C \mid e \text{ is maximal in } (A_C, \leq) \}.
\end{equation*}
By assumptions (a) and (b) we get that $D_C$ is a set of pairwise orthogonal, commuting idempotents. Furthermore, we claim that  that $\bigvee D_C$ is a set of local units for $S_C S_{C^{-1}}$. To this end, we define a function $\mu \colon A_C \to D_C$ by sending $m \in A_C$ to an element $\mu(m) \in D_C$ such that $m \leq \mu(m)$. Note that this is well-defined because of (c).

Now, let $z \in S_C S_{C^{-1}}$ be a general element. We shall construct an element $e \in \bigvee D_C$ such that $ez = z = ze$. Note that $z =\sum_i x_i y_i$  where the sum is finite and $x_i \in S_C$ and $y_i \in S_{C^{-1}}$ for each $i$. Next, if we consider a fixed $i$, $x_i$ decomposes as a finite sum $x_i = \sum s_g$ where $0 \ne s_g \in S_g$ for a finite number of elements $g \in C$. Similarly, $y_i = \sum t_h$ where $0 \ne t_h \in S_h$ for a finite number of elements $h \in C^{-1}$. More precisely, the sets,   
\begin{equation*}
 A = \bigcup_{i} \text{Supp}(x_i), \quad B = \bigcup_{i} \text{Supp}(y_i),
\label{eq:l7}
\end{equation*}
are both finite. By substituting the expressions for $x_i$ and $y_i$ in the definition of $z$ we see that $z$ is a finite sum of elements $s_g t_h$ with $g \in A$, $h \in B$, $s_g \in S_g$ and $t_h \in S_h$.

Take $g \in A, h \in B$ and let $a_g b_h$ be any element such that $a_g \in S_g, b_h \in S_h$. By Lemma \ref{lem:almost_module} there exist $m_1, m_2, \dots, m_i \in M_g \subseteq A_C$ and $m_1', m_2', \dots, m_j' \in M_{h^{-1}} \subseteq A_C$ such that, $(m_1 \vee m_2 \vee \dots \vee m_i)a_g = a_g$ and $b_h (m_1' \vee m_2' \vee \dots \vee m_j') = b_h$. Hence, 
\begin{equation*}
(m_1 \vee m_2 \vee \dots \vee m_i \vee m_1' \vee m_2' \vee \dots \vee m_j') a_g b_h = a_g b_h (m_1 \vee m_2 \vee \dots \vee m_i \vee m_1' \vee m_2' \vee \dots \vee m_j').
\end{equation*}
Let $f_{g} = \mu(m_1) \vee \mu(m_2) \vee \dots \vee \mu(m_i)$ and $f'_{h} = \mu(m_1') \vee \mu(m_2') \vee \dots \vee \mu(m_j')$. Note that $f_g \vee f'_h \in \bigvee D_C$ and $m \leq \mu(m)$ for any $m \in A_C$. Hence, 

\begin{equation}
a_g b_h = (f_g \vee f'_h) a_g b_h = a_g b_h (f_g \vee f'_h).
\label{eq:999}
\end{equation}
Let $e = \bigvee_{g \in A} f_g \vee \bigvee_{h \in B} f'_h$. Then, by (\ref{eq:999}) we get that $z = ez = ze$. 
\end{proof}

As a side note, we notice that the assumptions (a), (b) and (c) in Proposition \ref{prop:virtually_induced} are in fact satisfied in the special case of Leavitt path algebras. By studying the proof of Proposition \ref{prop:lpa_is_virtually} it can be seen that any element of $E_g$, for any $g \in G$, is of the form $\sum_i \alpha_i \alpha_i^*$ for some paths $\alpha_i$. It is straightforward to use (\ref{eq:1}) to show that any two elements of this form commute, hence (a) is satisfied. For any subset $A \subseteq G$ consider the maximal elements $M$ of the set $\bigvee_{g \in A} M_g$ with respect to the idempotent ordering. Again by using (\ref{eq:1}), we see that $M$ is the minimal elements with regards to the initial subpath ordering, i.e. $\alpha \alpha^* \leq \beta \beta^* $ if and only if $\alpha$ is an initial subpath of $\beta$. Hence, for any $\alpha \alpha^*, \beta \beta^* \in M$ such that $\alpha \ne \beta$, we have $(\alpha \alpha^*)(\beta \beta^*)=0$ by (\ref{eq:1}). Moreover, any given path $\alpha$ only has finitely many initial subpaths, i.e. any chain in the idempotent ordering containing $\alpha \alpha^*$ is finite. Thus (b) and (c) are satisfied in the case of Leavitt path algebras. Hence, as a corollary to Proposition \ref{prop:virtually_induced} and Proposition \ref{prop:lpa_is_virtually}, we obtain the following.

\begin{corollary}
Let $G$ be an arbitrary group, let $R$ be a unital ring and let $E$ be a directed graph. Consider the Leavitt path algebra $L_R(E)$ equipped with any standard $G$-grading. For every normal subgroup $N$ of $G$, the induced $G/N$-grading is virtually epsilon-strong. \label{cor:lpa_induced}
\end{corollary}
We now return to the main track and apply our investigation to the induced quotient group gradings of epsilon-strongly graded rings. The following theorem is one of our main results:

\begin{theorem}
Let $G$ be an arbitrary group and let $S=\bigoplus_{g \in G} S_g$ be an epsilon-strongly $G$-graded ring. 
Consider the induced $G/N$-grading of $S$. The following assertions hold:

\begin{enumerate}[(a)]
\begin{item}
For each class $C \in G/N$, the set $E_C = \bigvee \{ \epsilon_g \mid g \in C \}$ is a set of local units for the ring $S_C S_{C^{-1}}$. In particular, the induced $G/N$-grading of $S$ is essentially epsilon-strong;
\end{item}
\begin{item}
Suppose that for each class $C \in G/N$, (i) the poset $(E_C, \leq)$ contains no infinite chain and (ii) the maximal elements of $(E_C, \leq)$ are pairwise orthogonal. Then, the induced $G/N$-grading of $S$ is virtually epsilon-strong.
\end{item}
\end{enumerate}
\label{thm:main}
\end{theorem}
\begin{proof}

(a): Take an arbitrary $g \in G$. Since $S_g S_{g^{-1}}$ is a unital ring with multiplicative identity element $\epsilon_g$, it is a ring with a set of local units $E_g = \{ \epsilon_g \}$. Furthermore, $\epsilon_g \in Z(S_e)$ (see \cite[Prop. 5]{nystedt2016epsilon}). This means that $\epsilon_g \epsilon_h = \epsilon_h \epsilon_g$ for every $g, h \in G$. Hence, $\epsilon_g \vee \epsilon_h$ exists for every $g, h\in G$. Thus, by Proposition \ref{prop:almost_local_units}, for each $C \in G/N$, the ring $S_C S_{C^{-1}}$ has a set of local units,
\begin{equation*}
E_C = \bigvee_{g \in C} E_g = \bigvee \{ \epsilon_g \mid g \in C \}.
\end{equation*}


(b): Note that $S$ is virtually epsilon-strongly $G$-graded by letting $M_g = \{ \epsilon_g \}$ for each $g \in G$. Condition (a) in Proposition \ref{prop:virtually_induced} is satisfied since the $\epsilon_g$'s are central idempotents. Moreover, condition (b) and (c) are in this special case equivalent to $(E_C, \leq)$ having no infinite chains and $(E_C, \leq)$ having pairwise orthogonal maximal elements for each $C \in G/N$. 
\end{proof}

\begin{remark}We make no claim about the necessity of the conditions in Theorem \ref{thm:main}(b). Is it true that any induced quotient group grading of an arbitrary epsilon-strongly graded ring is virtually epsilon-strong? The author has not been able to find any example where the induced quotient grading is essentially epsilon-strong but not virtually epsilon-strong. The difference between `virtual' and `essential' seems subtle.
\label{rem:thm_rem}
\end{remark}

Finally, we establish our main result:
\begin{proof}[Proof of Theorem \ref{thm:main1}]
By Definition \ref{def:nystedt-gradings}, the induced $G/N$-grading is epsilon-strong if and only if (i) the induced $G/N$-grading is symmetric, and (ii) $S_C S_{C^{-1}}$ is unital for each $C \in G/N$. By Proposition \ref{prop:symmetric}, (i) holds. Moreover, by Proposition \ref{prop:greatest}, (ii) holds if and only if $E_C$ has an upper bound in $E(S_N)$ for each $C \in G/N$. 
\end{proof}

\section{Epsilon-finite gradings}
\label{sec:5}

In this section, we introduce a subclass of epsilon-strong $G$-gradings with the special property that any induced $G/N$-grading is epsilon-strong.

\begin{definition}
Let $G$ be an arbitrary group and let $S$ be an epsilon-strongly $G$-graded ring. Put $F=\bigvee \{ \epsilon_g \mid g \in G \}$. We call the $G$-grading \emph{epsilon-finite} if $F$ is a finite set. In that case, we say that $S$ is \emph{epsilon-finitely $G$-graded}.
\label{def:epsilon_finite}
\end{definition}

\begin{remark}
We make some remarks regarding Definition \ref{def:epsilon_finite}.
\begin{enumerate}[(a)]
\begin{item}
Note that the set $\{ \epsilon_g \mid g \in G \}$ is finite if and only if $\bigvee \{ \epsilon_g \mid g \in G\}$ is finite.
\end{item}
\begin{item}
A unital strongly $G$-graded ring $S=\bigoplus_{g \in G}S_g$ is epsilon-finite since $\epsilon_g = 1_S$ for every $g \in G$ (see \cite[Prop. 8]{nystedt2016epsilon}).
\end{item}
\begin{item}
An epsilon-strongly $G$-graded ring $S$ with finite support, i.e. $|\text{Supp}(S) | < \infty$, is necessarily epsilon-finite. However, the converse does not hold in general. Consider e.g. the complex group ring $\mathbb{C}[\mathbb{Z}]$, which is epsilon-finite but not finitely supported. 
\end{item}
\end{enumerate}
\label{rem:finite}
\end{remark}

\begin{proposition}
Let $S$ be an epsilon-finitely $G$-graded ring. For any normal subgroup $N$ of $G$, the induced $G/N$-grading is epsilon-finite.
\label{prop:epsilon_finite}
\end{proposition}
\begin{proof}
We begin by showing that the induced $G/N$-grading is epsilon-strong. By Theorem \ref{thm:main}(a), the induced $G/N$-grading is essentially epsilon-strong and it holds for every $C \in G/N$ that $E_C=\bigvee \{ \epsilon_g \mid g \in C \}$ is a set of local units for $S_C S_{C^{-1}}$.  To prove that the induced $G/N$-grading is epsilon-strong, we need to show that $S_C S_{C^{-1}}$ is unital for each class $C \in G/N$. Fix an arbitrary class $C \in G/N$. Then, $E_C=\bigvee \{ \epsilon_g \mid g \in C \} \subseteq \bigvee \{ \epsilon_g \mid g \in G \}$, where the latter set is finite by assumption. Hence, $E_C$ is finite and by Corollary \ref{cor:finite}, $E_C$ has a greatest element $\epsilon_C$ which is the multiplicative identity element of $S_C S_{C^{-1}}$ by Proposition \ref{prop:greatest}. Thus, the induced $G/N$-grading is epsilon-strong. 

Let $C \in G/N$ again be an arbitrary class. Note that $\epsilon_C = \bigvee_{g \in C} \epsilon_g = \bigvee_{i \in I_C} \epsilon_i$ for some finite set $I_C \subseteq C$ which completely determines $\epsilon_C$.  But, $\{ \epsilon_i \mid i \in \bigcup_{C \in G/N} I_C \} \subseteq \{ \epsilon_g \mid g \in G \}$ is a finite set. Thus, $\{ \epsilon_C \mid C \in G/N \}$ is finite.
\end{proof}

\begin{remark}At this point we make two remarks.
\begin{enumerate}[(a)]
\begin{item}
By Proposition \ref{prop:epsilon_finite}, the functor $U_{G/N}$ restricts to the category of epsilon-finitely $G$-graded rings.
\end{item}
\begin{item}
In the next section, we we will give an example  of an epsilon-strongly $G$-graded ring $(S, \{ S_g \}_{g \in G}) \in \text{ob}(G \mhyphen \epsilon \mathrm{STRG})$ which is not epsilon-finite but for which there is a normal subgroup $N$ of $G$ such that $(S, \{ S_C \}_{C \in G/N}) \in \text{ob}(G/N \mhyphen  \mathrm{STRG})$ (Example \ref{ex:2}). In particularly, note that $(S, \{ S_C \}_{C \in G/N})$ is epsilon-finite by Remark \ref{rem:finite}(b).
\end{item}
\end{enumerate}
\label{rem:finite_rem}
\end{remark}

We continue by proving that having a noetherian principal component is a sufficient condition for a grading to be epsilon-finite. Recall (see e.g. \cite[Ch. 7.22]{lam2001first}) that a ring $R$ is called \textit{block decomposable} if $1 \in R$ decomposes into $1 = c_1 + c_2 + \dots + c_r$ where $c_i$ are central primitive idempotents. If this holds then any central idempotent can be expressed as a finite sum $c = \sum_i c_i$ where the sum is taken over all $i$ such that $c c_i \ne 0$. In particular, this implies that the set of central idempotents is finite. 

\begin{theorem}
Let $S$ be an epsilon-strongly $G$-graded ring. The following assertions hold:
\begin{enumerate}[(a)]
\begin{item}
If $S_e$ is block decomposable, then $S$ is epsilon-finitely $G$-graded;
\end{item}
\begin{item}
If $S_e$ is left or right noetherian, then $S$ is epsilon-finitely $G$-graded.
\end{item}
\end{enumerate}
Hence, in particular, if $S_e$ is left or right noetherian and $N$ is any normal subgroup of $G$, then the induced $G/N$-grading of $S$ is epsilon-strong.
\label{thm:induced_epsilon}
\end{theorem}
\begin{proof}
(a): Note that $\epsilon_g$ is a central idempotent of $R$ (see \cite[Prop. 5]{nystedt2016epsilon}). This implies that ${\bigvee \{ \epsilon_g \mid g \in G \}}$ is a subset of the set of central idempotents of $R$. But since $R$ is assumed to be block decomposable, there are only finitely many central idempotents in $R$. Hence, the $G$-grading is epsilon-finite.  

(b): By \cite[Prop. 22.2]{lam2001first}, one-sided noetherianity is a sufficient condition for $R$ to be block decomposable. 
\end{proof}

\section{Examples}
\label{sec:6}

The class of unital partial skew group rings is an important type of rings with a natural epsilon-strong group grading. In this section, we will consider two concrete examples of unital partial skew group rings. Throughout, let $G$ be an arbitrary group with neutral element $e$ and let $R$ be an arbitrary ring equipped with a multiplicative identity element.

Partial actions of groups were first introduced in the study of $C^*$-algebras by Exel \cite{exel1994circle}. A later development by Dokuchaev and Exel \cite{dokuchaev2005associativity} was to consider partial actions on a ring in a purely algebraic context. Given a partial action on a ring $R$, they constructed the \emph{partial skew group ring} of the partial action, generalizing the classical skew group ring of an action on a ring. The partial skew group ring of a general partial action is not necessarily associative (cf. \cite[Expl. 3.5]{dokuchaev2005associativity}). However, the subclass of \emph{unital partial actions} (see Definition \ref{def:partial_action}) give rise to associative partial skew group rings. In fact, it is enough to assume that the partial action is idempotent (see \cite[Cor. 3.2]{dokuchaev2005associativity}).

\begin{definition}
A \emph{unital partial action} of $G$ on $R$ is a collection of unital ideals $\{ D_g \}_{g \in G}$ and ring isomorphisms $\alpha_g \colon D_{g^{-1}} \to D_g$ such that, 
\begin{enumerate}
\begin{item}
$D_e = R$ and $\alpha_e = id_R.$
\end{item}
\begin{item}
$\alpha_g(D_{g^{-1}} D_h) = D_g D_{gh}$ for all $g, h \in G$.
\end{item}
\begin{item}
For all $x \in D_{h^{-1}} D_{(gh)^{-1}}$, $\alpha_g(\alpha_h(x))=\alpha_{gh}(x)$.
\end{item}
\end{enumerate}
\label{def:partial_action}
\end{definition}
We let $1_g$ denote the multiplicative identity element of the ideal $D_g$. Note that the $1_g$'s are central idempotents in $R$. 

\smallskip

Recall (cf. e.g. \cite[Sect. 5]{lannstrom2018chain}) that a unital partial action $\alpha$ of a group $G$ on a ring $R$ gives us a unital, associative algebra called the \emph{unital partial skew group ring} $R \star_\alpha G = \bigoplus_{g \in G} D_g \delta_g$, where the $\delta_g$'s are formal symbols.  The multiplication is defined by linearly extending the relations,
\begin{equation*}
(a_g \delta_g)(b_h \delta_h) = a_g \alpha_g( b_h 1_{g^{-1}}) \delta_{gh},
\label{eq:partial_multiplication}
\end{equation*} 
for $g,h \in G$ and $a_g \in D_g, b_h \in D_g$. The relations in Definition \ref{def:partial_action} essentially tell us that this multiplication is well-defined. Moreover, $R \star_\alpha G$ is epsilon-strongly $G$-graded with $\epsilon_g = 1_g \delta_0$ (see \cite[pg. 2]{nystedt2016epsilon}).

\smallskip

Next, we will give an example of a unital partial action such that the unital partial skew group ring has a quotient grading which is not epsilon-strong.
\begin{example}

Let $R:=\text{Fun}(\mathbb{Z} \to \mathbb{Q})=\prod_{\mathbb{Z}} \mathbb{Q}$ be the algebra of bi-infinite sequences with component-wise addition and multiplication. 
Define a partial action of $(\mathbb{Z},+,0)$ on $R$ in the following way. Let $D_0 = R$ and $D_i = e_i R$, $i \ne 0$ where $e_i$ is the Kronecker delta sequence. More precisely, $e_i(j) = \delta_{i,j}$ for all $i,j \in \mathbb{Z}$. Moreover, define $\alpha_i \colon D_{-i} \to D_i$ by,
$$ f \mapsto (i \mapsto f(-i) ).$$

Note that $D_g D_h = (0)$ if $g,h \ne 0$ and $g \ne h$. Moreover, $D_g D_0 = D_0 D_g = D_g$ for all $g \in G$. This means, condition (2) in Definition \ref{def:partial_action} is satisfied for the case $g \ne 0, h \ne -g$. But on the other hand if $g \ne 0, h = -g$ then the condition reads $\alpha_g(D_{-g} D_{-g}) = D_{g} D_0$ which holds since $\alpha_g$ is an isomorphism. The case $g=0$ is also clear. Hence, condition (2) holds. Next consider condition (3). If $h \ne 0, g \ne - h$, then $D_{-h}D_{-g-h} = (0)$ hence the condition is trivially satisfied. On the other hand, if $h \ne 0, g = -h$, then the condition reads $\alpha_{-h}(\alpha_h(x)) = \alpha_0(x)$ for all $x \in D_{-h}D_0 = D_{-h}$, which also holds for our definition of $\alpha$. Finally, if $h=0$, then condition (3) just reads $\alpha_g(x) = \alpha_g(x)$ for all $x \in D_{-g}$ since $\alpha_0=\text{id}$ by definition. This proves that $\alpha$ is a unital partial action.

Consider the partial skew group ring $R \star_\alpha \mathbb{Z}$ with the canonical grading ${R \star_\alpha \mathbb{Z} = \bigoplus_{n \in \mathbb{Z}} D_n \delta_n}$. This grading is epsilon-strong with $\epsilon_i = e_i\delta_0$ (see \cite[pg. 1]{nystedt2016epsilon}). In particular, note that the set $ \{ \epsilon_i \mid i \in \mathbb{Z} \}$ is infinite and contains infinitely many orthogonal central idempotents.

Next, note that the induced $\mathbb{Z}/2\mathbb{Z}$-grading on $R \star_\alpha \mathbb{Z}$ has components, $$S_0 = \bigoplus_{n \in \mathbb{Z}} D_{2n} \delta_{2n}, \quad S_1 = \bigoplus_{n \in \mathbb{Z}} D_{2n+1} \delta_{2n+1}.$$ We will show that $S_1 S_1$ does not admit a multiplicative identity element, which implies that the induced $\mathbb{Z}/2\mathbb{Z}$-grading is not epsilon-strong.

Now, by Theorem \ref{thm:main}(a),  $$E = \bigvee \{ \epsilon_{2n+1} \mid n \in \mathbb{Z} \},$$ is a set of local units for $S_1 S_1$. By Proposition \ref{prop:greatest}, $S_1 S_1$ is unital if and only if $E$ has an upper bound in $E(S_1 S_1)\subseteq S_1S_1$. A moments thought gives that, if such an element exists, it must be $d \delta_0$ where $d \in R$, $d(2n+1)=1$ and $d(2n)=0$ for all $n \in \mathbb{Z}$. To conclude we prove that $ d \delta_0 \not \in S_1 S_1$. Note that any element $x \delta_0 \in S_1 S_1$ is a finite sum of elements of the form, $$(a \delta_{2n+1}) (b \delta_{-2n-1}) = \alpha_{2n+1}(\alpha_{-2n-1}(a) b) \delta_0 = a_i b_i' \delta_0,$$ where $a b' \in R$ is a function satisfying $\text{Supp}(a b')=\{2n+1\}$. This implies that, $\text{Supp}(x) < \infty$ for any element of the form $x \delta_0 \in S_1 S_1$. On the other hand, $\text{Supp}(d) = \infty$. Hence, $d \delta_0 \not \in S_1 S_1$. Thus, the induced $\mathbb{Z}/2 \mathbb{Z}$-grading is not epsilon-strong.

\label{ex:1}
\end{example}
\begin{remark}
Note that Example \ref{ex:1} shows that the functor $U_{G/N}$ does not restrict to the category $G \mhyphen \epsilon \textrm{STRG}$ (see Question \ref{question:2}). Also note that the induced quotient group grading in Example \ref{ex:1} is an example of a grading that is essentially epsilon-strong (in fact, virtually epsilon-strong) but not epsilon-strong (cf. Theorem \ref{thm:main}).
\label{rem:example_rem}
\end{remark}

The next example shows that being epsilon-finite is not a necessary condition for the induced quotient group grading to be epsilon-finite (cf. Remark \ref{rem:finite_rem}(b)).
\begin{example}

Let the group $(\mathbb{Z}, +,0)$ act (globally) on $R=\text{Fun}(\mathbb{Z} \to \mathbb{Q})$ by bilateral shifting. That is, $1$ acts by mapping,
\begin{equation*}
\dots \quad  a_{-3}  \quad  a_{-2}  \quad  a_{-1} \quad  \underline{ a_0} \quad  a_1 \quad  a_2 \quad  a_3  \quad \dots
\end{equation*} 
to
\begin{equation*}
\dots \quad  a_{-4}  \quad  a_{-3}  \quad  a_{-2} \quad  \underline{a_{-1}} \quad  a_0 \quad  a_1 \quad  a_2  \quad \dots
,
\end{equation*} 
and $-1$ similarly shifts the sequence one step in the other direction. 

Let $\beta \colon \mathbb{Z} \to \text{Aut}(R)$ be the group homomorphism corresponding to this action. By restricting the (global) action to an ideal of $R$ we get a partial action (see e.g. \cite[Section 4]{dokuchaev2005associativity}).
More precisely, let $A:= \{ f \in R \mid f(0) = 0 \}$. Then $A$ is a unital ideal of $R$. Moreover, the natural partial action $\{ \alpha_n \}_{n \in \mathbb{Z}}$ is defined on the ideals, $$D_n := A \cap \beta_n(A) = \{ f \in R \mid f(0)=f(n)=0 \},$$ for $n \in \mathbb{Z}$. For example $\alpha_2 \colon D_{-2} \to D_2$ maps,
\begin{equation*}
\dots \quad  a_{-3}  \quad  0  \quad  a_{-1} \quad  \underline{ 0} \quad  a_1 \quad  a_2 \quad  a_3  \quad \dots
\end{equation*} 
to
\begin{equation*}
\dots \quad  a_{-5}  \quad  a_{-4}  \quad  a_{-3} \quad  \underline{ 0} \quad  a_{-1} \quad  0 \quad  a_1  \quad \dots
\end{equation*} 


Consider the partial skew group ring $A \star_\alpha \mathbb{Z}=\bigoplus_{n \in \mathbb{Z}} D_n \delta_n$. For $n \in \mathbb{Z}$, the multiplicative identity element of $D_n$ is given by, 
\begin{equation*}
1_n(i) = \begin{cases} 0 & i=0 \\ 0 & i =n \\ 1 & \text{otherwise} \end{cases}
\end{equation*} for all $i,j \in \mathbb{Z}$. Note that the set $\{ \epsilon_n \mid n \in \mathbb{Z} \} = \{ 1_n \delta_0 \mid n \in \mathbb{Z} \}$ is infinite but the epsilons are not orthogonal!

Next, consider the induced $\mathbb{Z}/2\mathbb{Z}$-grading of $A \star_\alpha \mathbb{Z} = S_0 \oplus S_1$ where $S_0 = \bigoplus_{n \in \mathbb{Z}} D_{2n} \delta_{2n}$ and $S_1 = \bigoplus_{n \in \mathbb{Z} } D_{2n+1} \delta_{2n + 1} .$ Note that $S_0$ is a subring of $A \star_\alpha \mathbb{Z}$ with multiplicative identity $1_0 \delta_0$. We will show that $1_0 \delta_0 \in S_1 S_1$, proving that the induced grading is strong (see e.g. \cite[Prop. 1.1.1]{nastasescu2004methods}). Indeed, let $\gamma(i) = 0$ for $i \ne -1$ and $\gamma(-1)=1$. Then, $$(\gamma \delta_{-3})(1_{3}\delta_{3}) + (1_{-1}\delta_{-1})(1_1 \delta_1) = \gamma \delta_0 + 1_{-1} \delta_0 = 1_0 \delta_0.$$ Thus, the induced $\mathbb{Z}/2\mathbb{Z}$-grading is strong.

\label{ex:2}
\end{example}

\section{Induced quotient group gradings and epsilon-crossed products}
\label{sec:7}

The class of epsilon-crossed product is defined analogously to the classical algebraic crossed product (see \cite[Def. 32]{nystedt2016epsilon}). Moreover, the category of epsilon-crossed products is denoted by $G \mhyphen \epsilon \mathrm{CROSS}$ (see Section \ref{sub:epsilon_cross}). On the other hand, recall (see \cite[pg. 1]{nystedt2016epsilon}) the  notion of a \emph{unital partial crossed product}. This construction is a priori unrelated to the epsilon-crossed product and generalizes the classical algebraic crossed product by a twisted action. The full definition is rather technical. However, for our purposes it suffices to note that a unital partial skew group (see Definition \ref{def:partial_action}) is a special type of unital partial crossed product. 
The relationship between epsilon-crossed products and unital partial crossed products is described in the following theorem by Nystedt, Öinert and Pinedo:

\begin{theorem}
(\cite[Thm. 33]{nystedt2016epsilon})
If $(S, \{ S_g \}_{g \in G}) \in \textnormal{ob}(G \mhyphen \epsilon \mathrm{CROSS})$, then $S$ can be presented as a unital partial crossed product by a unital twisted partial action of $G$ on $S_e$. Conversely, if $S = \bigoplus_{g \in G}  D_g \delta_g$ is a unital partial crossed product, then $(S, \{ D_g \delta_g \}_{g \in G}) \in \textnormal{ob}(G \mhyphen \epsilon \mathrm{CROSS})$.
\label{thm:epsilon_crossed}
\end{theorem}

\vspace{-1em}
With this characterization in mind, we will now consider Question \ref{question:1}. Unfortunately, it does not hold in general that the induced $G/N$-grading of an epsilon-crossed product gives an epsilon-crossed product. The following example shows that the functor $U_{G/N}$ does not restrict to $G \mhyphen \epsilon \mathrm{CROSS}$. 

\begin{example}
Consider the unital partial skew group ring $(R \star_\alpha \mathbb{Z}, \{ D_i \delta_i \}_{i \in \mathbb{Z}})$ given in Example \ref{ex:1}. Note that $(R \star_\alpha \mathbb{Z}, \{ D_i \delta_i \}_{g \in \mathbb{Z}})$ is an epsilon-crossed product by Theorem \ref{thm:epsilon_crossed}. However, by Example \ref{ex:1}, the induced $\mathbb{Z}/2\mathbb{Z}$-grading $U_{\mathbb{Z} / 2\mathbb{Z}}((R \star_\alpha \mathbb{Z}, \{ D_i \delta_i \}_{i \in \mathbb{Z}})) \not \in \text{ob}(\mathbb{Z} / 2 \mathbb{Z} \mhyphen \epsilon \mathrm{STRG})$. Thus, in particular, $U_{\mathbb{Z} / 2\mathbb{Z}}((R \star_\alpha \mathbb{Z}, \{ D_i \delta_i \}_{i \in \mathbb{Z}})) \not \in \text{ob}(\mathbb{Z} / 2 \mathbb{Z} \mhyphen \epsilon \mathrm{CROSS}).$
\end{example}

The following proposition will allow us to find examples where the condition in Question \ref{question:1} is true.

\begin{proposition}
Let $G$ be an arbitrary group, let $R$ be a unital ring and let $\{\alpha_g \}_{g \in G}$ be a unital partial action on $R$ where $1_g$ denotes the multiplicative identity element of the ideal $D_g$. Assume that,
\begin{enumerate}[(a)]
\begin{item}
$D_g D_h = (0)$ for all $g,h \in G$ such that $g \ne h$ and $g,h \ne e$,
\end{item}
\begin{item}
the set $\{ 1_g \mid g \in G \}$ is finite. 
\end{item}
\end{enumerate}

Let $N$ be any normal subgroup of $G$. Then the induced $G/N$-grading of $R \star_\alpha G$ gives an epsilon-crossed product. In other words, $U_{G/N}((R \star_\alpha G, \{ D_g \delta_g \}_{g \in G})) \in \textnormal{ob}(G/N \mhyphen \epsilon \mathrm{CROSS})$.
\label{prop:partial_crossed}
\end{proposition}
\begin{proof}
We first show that the $G/N$-grading of $R \star_\alpha G$ is epsilon-strong. Write $R \star_\alpha G = \bigoplus_{C \in G/N} S_C$ and take a class $C \in G/N$. By Theorem \ref{thm:main}, we need to show that $E_C = \bigvee \{ 1_g \delta_0 \mid g \in C \}$ has an upper bound. Note that assumption (a) implies that $1_g 1_h=0$ for $g \ne h \ne e$. Hence, by (b) and Proposition \ref{prop:greatest}, $$\epsilon_C = \bigvee_{ g \in C } 1_g \delta_0 = 1_{g_1} \delta_0 + 1_{g_2} \delta_0 + \dots 1_{g_n} \delta_0, $$ for some choice of representatives $g_i \in C$ for $1 \leq i \leq n$ satisfying, $$\{ 1_{g_i } \mid 1 \leq i \leq n \} = \{ 1_g \mid g \in C \}.$$ Hence, $S_C S_{C^{-1}}$ is unital. Since $C \in G/N$ was chosen arbitrarily, it follows that the induced $G/N$-grading of $R \star_\alpha G$ is epsilon-strong.

Next, we show that for each $C \in G/N$ there is an epsilon-invertible element in $S_C$. For $g, h \in G$ and $x \in D_g, y \in D_h$, we have that $(x \delta_g)(y \delta_h) = \alpha_{g}(\alpha_{g^{-1}}(x)y) \delta_{gh} = 0$ if $h \ne g^{-1}$ since $D_{g^{-1}} D_h = (0)$ by assumption. Moreover, it is easy to check that $1_{g_i} \delta_0 = (1_{g_i} \delta_{g_i})(1_{g_i^{-1}} \delta_{g_i^{-1}})$ for $1 \leq i \leq n$. Let $s= ( 1_{g_1} \delta_{g_1} + \dots + 1_{g_n} \delta_{g_n}) \in S_C $ and $t = (1_{g_1^{-1}} \delta_{g_1^{-1}} + \dots + 1_{g_n^{-1}} \delta_{g_n^{-1}}) \in S_{C^{-1}}$. Then, 
\begin{equation*}
st = \sum_{1 \leq i,j \leq n} (1_{g_i} \delta_{g_i})(1_{g_j^{-1}} \delta_{g_j^{-1}}) = \sum_{1 \leq i \leq n} (1_{g_i} \delta_{g_i})(1_{g_i^{-1}} \delta_{g_i^{-1}}) = \sum_{1 \leq i \leq n} 1_{g_i} \delta_0 = \epsilon_C.
\end{equation*}
Thus, $(R \star_\alpha G, \{ S_C \}_{C \in G/N})=U_{G/N}((R \star_\alpha G, \{ D_g \delta_g \}_{g \in G}))$ is an epsilon-crossed product.
\end{proof}

It is not clear to the author if the conditions in Proposition \ref{prop:partial_crossed} are necessary. The following is an explicit example where the induced quotient group grading gives an epsilon-crossed product. In other words, an example of a unital partial skew group ring (which is an epsilon-crossed product by Theorem \ref{thm:epsilon_crossed}) such that the condition in Question \ref{question:1} is true.
\begin{example}
Let $R=\mathbb{Q} \times \mathbb{Q} \times \mathbb{Q} \times \mathbb{Q}$ and let $G$ be the cyclic group of order 4. Then $G$ acts on $R$ by shifting. Let $e_i$ denote the Kronecker sequence, i.e. $e_i(j) = \delta_{i,j}$ for all integers $i,j$. Consider the unital ideal $A=e_1R + e_2R$ and the induced partial action on $A$. We get that $D_0=A, D_1=e_2 \mathbb{Q}, D_2=0, D_3=e_1\mathbb{Q}$.  Now, consider the unital partial skew group ring $A \star_\alpha G=D_0 \delta_0 \oplus D_1\delta_1 \oplus D_2 \delta_2 \oplus D_3 \delta_3.$ 

Let $N$ be the cyclic group of order 2. By Proposition \ref{prop:partial_crossed}, the induced $G/N$-grading gives an epsilon-crossed product. That is, $U_{G/N}((R \star_\alpha G, \{ D_g \delta_g \}_{g \in G})) \in \text{ob}(G/N \mhyphen \epsilon \mathrm{CROSS})$.  Note that, $$S_{[0]} = D_0 \delta_0 \oplus D_2 \delta_2 = D_0 \delta_0, \quad S_{[1]} = D_1 \delta_1 \oplus D_3 \delta_3. $$ Furthermore, $$\epsilon_{[1]}=(1,1)\delta_0=( (0,1)\delta_1 + (1,0) \delta_3) ( (0,1) \delta_1 + (1,0) \delta_3) ,$$ where ${(0,1)\delta_1 + (1,0)\delta_3 \in S_{[1]}}$ is an epsilon-invertible element.     

\end{example}

\section*{acknowledgement}
This research was partially supported by the Crafoord Foundation (grant no. 20170843). The author is grateful to Stefan Wagner, Johan Öinert and Patrik Nystedt for giving feedback and comments that helped to improve this manuscript. The author is also grateful to an anonymous referee for giving useful comments.

\bibliographystyle{abbrv}
\bibliography{induced_epsilon-strong}

\begin{thebibliography}{10}

\bibitem{abrams2017leavitt}
G.~Abrams, P.~Ara, and M.~S. Molina.
\newblock {\em Leavitt path algebras}, volume 2191.
\newblock Springer, 2017.

\bibitem{abrams2005leavitt}
G.~Abrams and G.~A. Pino.
\newblock The {L}eavitt path algebra of a graph.
\newblock {\em Journal of Algebra}, 293(2):319--334, 2005.

\bibitem{abrams1983morita}
G.~D. Abrams.
\newblock Morita equivalence for rings with local units.
\newblock {\em Communications in Algebra}, 11(8):801--837, 1983.

\bibitem{ara2007nonstable}
P.~Ara, M.~A. Moreno, and E.~Pardo.
\newblock Nonstable {K}-theory for graph algebras.
\newblock {\em Algebras and representation theory}, 10(2):157--178, 2007.

\bibitem{clark2018generalized}
L.~O. Clark, R.~Exel, and E.~Pardo.
\newblock A generalized uniqueness theorem and the graded ideal structure of
  {S}teinberg algebras.
\newblock In {\em Forum Mathematicum}, volume~30, pages 533--552. De Gruyter,
  2018.

\bibitem{dade1980group}
E.~C. Dade.
\newblock Group-graded rings and modules.
\newblock {\em Mathematische Zeitschrift}, 174(3):241--262, 1980.

\bibitem{dokuchaev2005associativity}
M.~Dokuchaev and R.~Exel.
\newblock Associativity of crossed products by partial actions, enveloping
  actions and partial representations.
\newblock {\em Transactions of the American Mathematical Society},
  357(5):1931--1952, 2005.

\bibitem{exel1994circle}
R.~Exel.
\newblock Circle actions on {C*}-algebras, partial automorphisms, and a
  generalized {Pimsner-Voiculescu} exact sequence.
\newblock {\em Journal of functional analysis}, 122(2):361--401, 1994.

\bibitem{hazrat2013graded}
R.~Hazrat.
\newblock The graded structure of {L}eavitt path algebras.
\newblock {\em Israel Journal of Mathematics}, 195(2):833--895, 2013.

\bibitem{johnson2012commutative}
B.~P. Johnson.
\newblock {\em Commutative rings graded by abelian groups}.
\newblock PhD Thesis. The University of Nebraska-Lincoln, 2012.

\bibitem{lam2001first}
T.~Lam.
\newblock {\em A First Course in Noncommutative Rings}.
\newblock Graduate Texts in Mathematics. Springer New York, 2001.

\bibitem{lannstrom2018chain}
D.~L{\"a}nnstr{\"o}m.
\newblock Chain conditions for epsilon-strongly graded rings with applications
  to {L}eavitt path algebras.
\newblock {\em arXiv preprint arXiv:1808.10163}, 2018.

\bibitem{nastasescu2004methods}
C.~Nastasescu and F.~van Oystaeyen.
\newblock {\em Methods of Graded Rings}.
\newblock Lecture Notes in Mathematics. Springer, 2004.

\bibitem{nystedt2018epsilon}
P.~Nystedt.
\newblock Epsilon-strong systems, partial inverse semigroup rings and steinberg
  algebras.
\newblock {\em arXiv preprint arXiv:1805.11955}, 2018.

\bibitem{nystedt2018unital}
P.~Nystedt.
\newblock Unital rings, rings with enough idempotents, rings with sets of local
  units, locally unital rings, s-unital rings and idempotent rings.
\newblock {\em arXiv preprint arXiv:1809.02117}, 2018.

\bibitem{nystedt2016epsilon}
P.~Nystedt, J.~Öinert, and H.~Pinedo.
\newblock Epsilon-strongly graded rings, separability and semisimplicity.
\newblock {\em Journal of Algebra}, 514:1 -- 24, 2018.

\bibitem{nystedt2017epsilon}
P.~Nystedt and J.~{\"O}inert.
\newblock Leavitt path algebras are nearly epsilon-strongly graded.
\newblock {\em arXiv preprint arXiv:1703.10601v2}, 2018.

\bibitem{sehgal2003graded}
S.~Sehgal and M.~Zaicev.
\newblock Graded identities and induced gradings on group algebras.
\newblock In {\em Groups, rings, Lie and Hopf algebras}, pages 211--219.
  Springer, 2003.

\bibitem{tomforde2011leavitt}
M.~Tomforde.
\newblock Leavitt path algebras with coefficients in a commutative ring.
\newblock {\em Journal of Pure and Applied Algebra}, 215(4):471--484, 2011.

\bibitem{tominaga1976s}
H.~Tominaga.
\newblock On s-unital rings.
\newblock {\em Mathematical Journal of Okayama University}, 18(2), 1976.

\end{thebibliography}

\end{document}